\documentclass[11pt,reqno]{amsart}
\usepackage{amssymb}
\usepackage[all]{xy}
\usepackage{hyperref}
\usepackage{enumitem}
\newtheorem{thm}{Theorem}[section]
\newtheorem{lem}[thm]{Lemma}
\newtheorem{prop}[thm]{Proposition}
\newtheorem{cor}[thm]{Corollary}

\theoremstyle{definition}
\newtheorem{dfn}[thm]{Definition}
\newtheorem{ex}[thm]{Example}

\theoremstyle{remark}
\newtheorem{remark}[thm]{Remark}

\newcommand{\CA}{{\mathcal{A}}}
\newcommand{\CD}{{\mathcal{D}}}

\newcommand{\CF}{{\mathcal{F}}}
\newcommand{\CH}{{\mathcal{H}}}
\newcommand{\CS}{{\mathcal{S}}}
\newcommand{\CI}{{\mathcal{I}}}
\newcommand{\CJ}{{\mathcal{J}}}
\newcommand{\CK}{{\mathcal{K}}}
\newcommand{\CL}{{\mathcal{L}}}

\newcommand{\CT}{{\mathcal{T}}}

\newcommand{\CB}{{\mathcal{B}}}
\newcommand{\CO}{{\mathcal{O}}}

\newcommand{\CR}{{\mathcal{R}}}
\newcommand{\CX}{{\mathcal{X}}}

\newcommand{\af}{\alpha}
\newcommand{\bt}{\beta}
\newcommand{\gm}{\gamma}

\newcommand{\ld}{\lambda}

\newcommand{\C}{{\mathbb{C}}}
\newcommand{\N}{{\mathbb{N}}}
\newcommand{\T}{{\mathbb{T}}}

\begin{document}


\title[Gauge-invariant ideals of $C^*$-algebras of Boolean dynamical
systems]
{Gauge-invariant ideals of $C^*$-algebras of Boolean dynamical
systems}

\author[T. M. Carlsen]{Toke Meier Carlsen}
\address{
Department of Sciences and Technology \\
University of Faroe Islands, N\'oat\'un 3, FO-100 T\'orshavn
\\
Faroe Islands} \email{toke.carlsen\-@\-gmail.\-com }

\author[E. J. Kang]{Eun Ji Kang$^{\dagger}$}
\thanks{Research partially supported by NRF-2017R1D1A1B03030540$^{\dagger}$}
\address{
Research Institute of Mathematics, Seoul National University, Seoul 08826, 
Korea} \email{kkang33\-@\-snu.\-ac.\-kr }




\begin{abstract}

We enlarge the class of $C^*$-algebras of Boolean dynamical systems in
order to include all weakly left-resolving normal labelled space
$C^*$-algebras in it. We prove a gauge-invariant uniqueness theorem
and classify all gauge-invariant ideals of these  $C^*$-algebras of generalized Boolean
dynamical systems and describe the corresponding
quotients as $C^*$-algebras of relative generalized Boolean dynamical
systems.

\end{abstract}

\maketitle

\setcounter{equation}{0}

\section{Introduction}

Inspired by the $C^*$-algebra of labelled graphs introduced in
\cite{BP1}, the class of $C^*$-algebras of Boolean dynamical systems
for which each action has compact range and closed domain was
introduced in \cite{COP}. In the setting of \cite{COP}, the class of
$C^*$-algebras of Boolean dynamical systems contains many labelled
graph $C^*$-algebras, but not all of them.

In this paper, we enlarge the class of $C^*$-algebras of Boolean
dynamical systems so that it contains all weakly left-resolving normal
labelled space $C^*$-algebras. As a result, this class of
$C^*$-algebras contains all graph $C^*$-algebras, all ultragraph
$C^*$-algebras, $C^*$-algebras of shift spaces and homeomorphism
$C^*$-algebras over 0-dimensional compact spaces, among others.

This is done by considering pairs consisting of a Boolean dynamical
system $(\CB,\CL,\theta)$ and a family $(\CI_\alpha)_{\alpha\in\CL}$
of ideals in $\CB$ such that $\theta_\alpha(\CB)\subseteq\CI_\alpha$
for each $\alpha$, and associate a universal $C^*$-algebra
$C^*(\CB,\CL,\theta,\CI_\alpha)$ to each such pair. We call such a
pair $(\CB,\CL,\theta,\CI_\af)$ a \emph{generalized Boolean dynamical
system}.

If $(\CB,\CL,\theta)$ is a Boolean dynamical system with compact range
$\CR_\alpha$ and closed domain as in \cite{COP} and we let
$\CI_\alpha=\{B\in\CB:B\subseteq \CR_\alpha\}$ for each $\alpha$, then
$C^*(\CB,\CL,\theta,\CI_\af)$ is canonically isomorphic to
$C^*(\CB,\CL,\theta)$. If $(E,\CL,\CB)$ is a weakly left-resolving
normal labelled space as in \cite{BaCaPa2017} where $(E,\CL)$ is a
labelled graph over $\CA$, and $C^*(E,\CL,\CB)$ is the $C^*$-algebra
associated with $(E,\CL,\CB)$ in \cite[Definition 2.5]{BaCaPa2017},
then $C^*(\CB,\CA,\theta,\CI_{r(\alpha)})$ is canonically isomorphic
to $C^*(E,\CL,\CB)$ where $\theta$ is the action of $\CA$ on $\CB$
given by $\theta_\alpha(A)=r(A,\alpha)$ and
$\CI_{r(\alpha)}=\{B\in\CB:B\subseteq r(\alpha)\}$.

The second goal of the paper is to give a description of the
gauge-invariant ideals of our $C^*$-algebras and thus generalize the
description of the gauge-invariant ideals of $C^*$-algebras of
set-finite, receiver set-finite and weakly left-resolving labelled
spaces given in \cite{JKP} and the description of the gauge-invariant
ideals of $C^*$-algebras of locally finite Boolean dynamical systems
given in \cite{COP}.

Working with gauge-invariant ideals of generalized Boolean dynamical
system, we found it convenient to use certain extensions of
$C^*$-algebras of generalized Boolean dynamical systems. Such an
extension is constructed from a generalized Boolean dynamical system
$(\CB,\CL,\theta,\CI_\af)$ toghether with an ideal $\CJ$ of $\CB_{reg}$.
We call such a system  a \emph{relative generalized Boolean
dynamical system}, and we associate a universal $C^*$-algebra
$C^*(\CB,\CL,\theta,\CI_\alpha;\CJ)$ to it. These $C^*$-algebras of
relative generalized Boolean dynamical systems are generalizations of
relative graph $C^*$-algebras introduced by Muhly and Tomforde in
\cite{MT2004}. If $\CJ=\CB_{reg}$, then
$C^*(\CB,\CL,\theta,\CI_\alpha;\CJ)=C^*(\CB,\CL,\theta,\CI_\alpha)$.

By imitating a construction in \cite{BaCaPa2017}, we show that if
$(\CB,\CL,\theta,\CI_\alpha;\CJ)$ is a relative generalized Boolean
dynamical system, then $C^*(\CB,\CL,\theta,\CI_\alpha;\CJ)$ can be
constructed as relative Cuntz--Pimsner algebra, and then use Katsura's
gauge-invariant uniqueness theorem for relative Cuntz--Pimsner
algebras \cite[Corollary 11.8]{Ka2007} to obtain a gauge-invariant
uniqueness theorem for $C^*(\CB,\CL,\theta,\CI_\alpha;\CJ)$. We then
use this to show that there is a one-to-one correspondences between
gauge-invariant ideals of $C^*(\CB,\CL,\theta,\CI_\alpha;\CJ)$ and
pairs $(\CH,\CS)$ where $\CH$ is a hereditary $\CJ$-saturated ideal of
$\CB$ and $\CS$ is an ideal of $\{A\in\CB:
\theta_\alpha(A)\in\CH\text{ for all but finitely many }\alpha\}$ such
that $\CH\cup\CJ\subseteq\CS$. We show in addition that the quotient
by the ideal corresponding to $(\CH,\CS)$ is isomorphic to the
$C^*$-algebra of relative generalized Boolean dynamical system that
can be constructed from $(\CB,\CL,\theta,\CI_\alpha;\CJ)$ and
$(\CH,\CS)$.

We also use a construction from \cite{Ka2007} to show that the
$C^*$-algebra $C^*(\CB,\CL,\theta,\CI_\alpha;\CJ)$ of the relative
generalized Boolean dynamical system $(\CB,\CL,\theta,\CI_\alpha;\CJ)$
is canonically isomorphic to the $C^*$-algebra
$C^*(\widetilde{\CB},\CL,\widetilde{\theta},\widetilde{\CI}_\alpha)$
of a generalized Boolean dynamical system
$(\widetilde{\CB},\CL,\widetilde{\theta},\widetilde{\CI}_\alpha)$, and
we show that if $(\CB,\CL,\theta)$ is a Boolean dynamical system,
$\CJ$ is an ideal of $\CB_{reg}$, and $(\CI_\alpha)_{\alpha\in\CL}$ is
a family of ideals of $\CB$ such that $\theta_\alpha(\CB)\subseteq
\CI_\alpha$ for each $\alpha$, then
$C^*(\CB,\CL,\theta,\CI_\alpha;\CJ)$ is a full hereditary
$C^*$-subalgebra of $C^*(\CB,\CL,\theta,\CB;\CJ)$. It follows that
$C^*(\CB,\CL,\theta,\CI_\alpha;\CJ)$ and
$C^*(\CB,\CL,\theta,\CR_\alpha;\CJ)$ where
$\CR_\alpha=\{A\in\CB:A\subseteq\theta_\alpha(B)\text{ for some
}B\in\CB\}$, are Morita equivalent. We call the latter $C^*$-algebra
the \emph{$C^*$-algebra of the relative Boolean dynamical system
$(\CB,\CL,\theta;\CJ) :=(\CB,\CL,\theta, \CR_\af;\CJ)$} or just the \emph{$C^*$-algebra of the Boolean
dynamical system $(\CB,\CL,\theta)$} if $\CJ=\CB_{reg}$. We thus have
that the $C^*$-algebra of any relative generalized Boolean dynamical
system is Morita equivalent to the $C^*$-algebra of a Boolean
dynamical system.

\subsection*{Pr\'ecis} 
This rest of the paper is organized as follows.

In section \hyperref[Preliminaries]{2}, we recall the notions of
Boolean algebras and Boolean dynamical systems. In section
\hyperref[GBDS]{3}, we introduce the definition of a generalized
Boolean dynamical system (Definition~\ref{def:GBDS}) and its
$C^*$-algebra. 

In section \hyperref[examples]{4}, we show that the
$C^*$-algebras of Boolean dynamical systems with compact range and
closed domain introduced in \cite{COP} and the $C^*$-algebras of
weakly left-resolving normal labelled spaces considered in
\cite{BaCaPa2017} are all $C^*$-algebras of generalized Boolean
dynamical systems (Example~\ref{BDS of cr-cd} and
Example~\ref{labelled graph}).

In section \hyperref[RGBDS is RCP]{5}, we construct from each relative
generalized Boolean dynamical system $(\CB,\CL,\theta, \CI_\af;\CJ)$,
a relative Cuntz-Pimsner algebra that is isomorphic to $C^*(\CB,\CL,
\theta, \CI_\af;\CJ)$ (Theorem~\ref{universal property:RGBDS}). As a
corollary, we get that  $C^*(\CB,\CL, \theta, \CI_\af)$ is isomorphic
to a Cuntz-Pimsner algebra (Corollary~\ref{existence of $C^*$-BDS}).

In section \hyperref[GIUT]{6}, we use Katsura's gauge-invariant
uniqueness theorem for relative Cuntz--Pimsner algebras
\cite[Corollary 11.8]{Ka2007} to obtain a gauge-invariant uniqueness
theorem for $C^*(\CB,\CL,\theta,\CI_\alpha;\CJ)$ (Theorem \ref{GIUT
for RGBDS}). As a corollary, a gauge-invariant uniqueness theorem for $C^*(\CB, \CL,\theta,\CI_\af)$ will also be given (Corollary~\ref{GIUT for GBDS}). 
We then show that $C^*(\CB,\CL,\theta,\CI_\alpha;\CJ)$ is a full hereditary $C^*$-subalgebra of $C^*(\CB,\CL,\theta,\CB;\CJ)$ (Proposition \ref{full hereditary}), and we construct a Boolean dynamical system $(\widetilde{\CB},\CL,\widetilde{\theta})$ such that $C^*(\CB,\CL,\theta,\CI_\alpha;\CJ)$ and $C^*(\widetilde{\CB},\CL,\widetilde{\theta},\widetilde{\CI}_\alpha)$ are isomorphic (Proposition \ref{RGBDS isom GBDS}).

In section \hyperref[gauge-invariant ideals]{7}, we classify the
gauge-invariant ideals of $C^*(\CB, \CL,\theta,\CI_\af;\CJ)$. We show
that a pair $(\CH,\CS)$ where $\CH$ is a hereditary and
$\CJ$-saturated ideal in $\CB$ and $\CS$ is an ideal of $\{A\in\CB:
\theta_\alpha(A)\in\CH\text{ for all but finitely many }\alpha\}$ such
that $\CH\cup\CJ\subseteq\CS$ give rises to a gauge-invariant ideal
$I_{(\CH,\CS)}$ of $C^*(\CB, \CL,\theta,\CI_\af;\CJ)$
(Lemma~\ref{Lemma:IHS}), and prove that the quotient $C^*(\CB,
\CL,\theta,\CI_\af;\CJ)/I_{(\CH,\CS)} $ of $C^*(\CB,
\CL,\theta,\CI_\af;\CJ)$ by the gauge-invariant ideal $I_{(\CH,\CS)}$
is canonically isomorphic to a relative Boolean $C^*$-algebra $C^*(\CB
/ \CH, \CL,\theta, [\CI_\af];[\CS])$
(Proposition~\ref{isomorphism to quotient}). Using this result, we
show that the gauge-invariant ideals of $C^*(\CB,
\CL,\theta,\CI_\af;\CJ)$ are in one-to-one correspondence with the
pairs $(\CH, \CS)$ (Theorem~\ref{gauge invariant
ideal:characterization}).

\section{Preliminaries}\label{Preliminaries}
In this section, we review the notions of Boolean algebras and Boolean
dynamical systems. For the most part we use the notational conventions
of \cite{COP}.

\subsection{Boolean algebras}

A {\em Boolean algebra} is a set $\CB$ with a distinguished element $\emptyset$ and maps $\cap: \CB \times \CB \rightarrow \CB$, $\cup: \CB \times \CB \rightarrow \CB$ and $\setminus: \CB \times \CB \rightarrow \CB$ such that $(\CB,\cap,\cup)$ is a distributive lattice, $A\cap \emptyset=\emptyset$ for all $A\in\CB$, and $(A\cap B)\cup (A\setminus B)=A$ and $(A\cap B)\cap (A\setminus B)=\emptyset$ for all $A,B\in\CB$. The Boolean algebra $\CB$ is called {\em unital} if there exists $1 \in \CB$ such that $1 \cup A = 1$ and $1 \cap A=A$ for all $A \in \CB$. (often, Boolean algebras are assumed to be unital and what we here call a Boolean algebra is often called a \emph{generalized Boolean algebra}).

We call $A\cup B$ the \emph{union} of $A$ and $B$, $A\cap B$ the \emph{intersection} of $A$ and $B$, and $A\setminus B$ the \emph{relative complement} of $B$ with respect to $A$. A subset $\CB' \subseteq \CB$ is called a {\em Boolean subalgebra} if $\emptyset\in\CB'$ and $\CB'$ is closed under taking union, intersection and the relative complement. A Boolean subalgebra of a Boolean algebra is itself a Boolean algebra.

We define a partial order on $\CB$ as follows: for $A,B \in \CB$,
\[
A \subseteq B ~~~\text{if and only if}~~~A \cap B =A.
\]
Then $(\CB, \subseteq)$ is a partially ordered set, and $A\cup B$ and $A\cap B$ are the least upper-bound and the greastest lower-bound of $A$ and $B$ with respect to the partial order $\subseteq$. If a family $\{A_{\ld}\}_{\ld \in \Lambda}$ of elements from $\CB$ has a least upper-bound, then we denote it by $\cup_{\ld \in \Lambda} A_\ld$. If $A\subseteq B$, then we say that $A$ is a \emph{subset of} $B$.

A non-empty subset $\CI$ of $\CB$ is called  an {\em ideal} \cite[Definition 2.4]{COP} if 
\begin{enumerate}
\item[(i)] if $A, B \in \CI$, then $A \cup B \in \CI$,
\item[(ii)] if $A \in \CI$ and $ B \in \CB$, then   $A \cap B \in \CI$. 
\end{enumerate}

An ideal $\CI$ of a Boolean algebra $\CB$ is a Boolean subalgebra. For $A \in \CB$, the ideal generated by $A$ is defined by $\CI_A:=\{ B \in \CB : B \subseteq A\}.$

\subsection{Boolean dynamical systems}

A map $\phi: \CB \rightarrow \CB'$ between two Boolean algebras is called a {\em Boolean homomorphism} (\cite[Definition 2.1]{COP}) if $\phi(A \cap B)=\phi(A) \cap \phi(B)$, $\phi(A \cup B)=\phi(A) \cup \phi(B)$, and $\phi(A \setminus B)=\phi(A) \setminus \phi(B)$ for all $A,B \in \CB$.

A map $\theta: \CB \rightarrow \CB $ is called an {\it action} (\cite[Definition 3.1]{COP})  on a Boolean algebra $\CB$ if it is a Boolean homomorphism with $\theta(\emptyset)=\emptyset$.

\vskip 0.5pc

Given a set $\CL$ and any $n \in \N$, we define $\CL^n:=\{(\af_1, \dots, \af_n): \af_i \in \CL\}$ and $\CL^*:=\cup_{n \geq 0} \CL^n$, where $\CL^0:=\{\emptyset \}$. For $\alpha\in\CL^n$, we write  $|\af|:=n$. For $\af=(\af_1, \dots, \af_n), \beta=(\beta_1,\dots,\beta_m) \in \CL^*$, we will usually write $\af_1 \dots \af_n$ instead of $(\af_1, \dots, \af_n)$ and use $\alpha\beta$ to denote the word $\af_1 \cdots \af_n\beta_1\dots\beta_m$ (if $\alpha=\emptyset$, then $\alpha\beta:=\beta$; and if $\beta=\emptyset$, then $\alpha\beta:=\alpha$). For $k\in\N$, we let $\alpha^k:=\alpha\alpha\dots\alpha$ where the concatenation on the right has $k$ terms. Similary we let $\alpha^0:=\emptyset$. For $1\leq i\leq j\leq |\af|$, we denote by $\af_{[i,j]}$ the sub-word $\af_i\cdots \af_j$ of  $\af=\af_1\af_2\cdots\af_{|\af|}$, where $\af_{[i,i]}=\af_i$.

\begin{dfn} 
A {\em Boolean dynamical system} is a triple $(\CB,\CL,\theta)$ where $\CB$ is a Boolean algebra, $\CL$ is a set, and $\{\theta_\af\}_{\af \in \CL}$ is a set of actions on $\CB$ such that for $\af=\af_1 \cdots \af_n \in \CL^*\setminus\{\emptyset\}$, the action $\theta_\af: \CB \rightarrow \CB$ is defined as $\theta_\af:=\theta_{\af_n} \circ \cdots \circ \theta_{\af_1}$.  We  also define $\theta_\emptyset:=\text{Id}$. 
\end{dfn}

\begin{remark}\label{compact range} 
Given a  Boolean algebra $\CB$, we say that an action $\theta$ on $\CB$ has {\em compact range} \cite[Definition 3.1]{COP} if $\{\theta(A)\}_{A \in \CB}$ has a least upper-bound. We denote by $\CR_{\theta}$ this least upper-bound if it exists. We say that an action $\theta$ has {\em closed domain} \cite[Definition 3.1]{COP} if there exists $\CD_{\theta} \in \CB$ such that $\theta(\CD_{\theta})=\CR_{\theta}$. In \cite[Definition 3.3]{COP}, a triple $(\CB,\CL,\theta)$ is called a Boolean dynamical system if  $\theta_\af$ has compact range $\CR_{\theta_\af}$ and closed domain $\CD_{\theta_\af}$  for each $\af \in \CL$. 

Notice that when we call $(\CB,\CL,\theta)$ a Boolean dynamical system in this paper,
 we do not assume that $\theta_\af$ has compact range and closed domain.
\end{remark}

For $B \in \CB$, we define
\[
\Delta_B^{(\CB,\CL,\theta)}:=\{\af \in \CL:\theta_\af(B) \neq
\emptyset \} ~\text{and}~  \ld_B^{(\CB,\CL,\theta)}:=|\Delta_B^{(\CB,\CL,\theta)}|.
\]
We will often just write $\Delta_B$ and $\ld_B$ instead of
$\Delta_B^{(\CB,\CL,\theta)}$ and $\ld_B^{(\CB,\CL,\theta)}$.

We say that $A \in \CB$ is {\em regular} (\cite[Definition 3.5]{COP})
if for any $\emptyset \neq B \in \CI_A$, we have $0 < \ld_B < \infty$.
If $A \in \CB$ is not regular, then it is called a {\em singular} set.
We write $\CB^{(\CB,\CL,\theta)}_{reg}$ or just $\CB_{reg}$ for the
set of all regular sets. Notice that $\emptyset\in\CB_{reg}$.

\subsection{Quotient Boolean dynamical systems}

Let $(\CB,\CL,\theta)$ be a Boolean dynamical system and let $\CJ$ be
an ideal of $\CB_{reg}$ and $\CH$ an ideal of $\CB$. As in \cite{COP},
we say that $\CH$ is {\em hereditary} if $\theta_{\af}(A) \in \CH$ for
$A \in \CH$ and $ \af \in \CL$, and we say that $\CH$ is {\em
$\CJ$-saturated} if $A \in \CH$ whenever $A \in \CJ$ and
$\theta_{\af}(A) \in \CH$ for all $\af \in \Delta_A$. When
$\CJ=\CB_{reg}$, then we often say {\em saturated} instead of
$\CJ$-saturated.

\vskip 1pc

If $\CH$ is an ideal of a Boolean algebra $\CB$, then the relation
\begin{eqnarray}\label{equivalent relation} 
A \sim B \iff \mathbf{R}: A \cup A' = B \cup B'~\text{for some}~ A', B' \in \CH
\end{eqnarray}
defines an equivalent relation on $\CB$ (\cite[Definition 2.5]{COP}).

\begin{remark} 
We have that $A\sim B$ if and only if either (and thus both) of the following two equivalent conditions hold.
\begin{enumerate}
\item[$\mathbf{R1}$] $A \cup A' = B \cup B' ~\text{for some}~ A', B' \in \CH ~\text{with}~  A \cap A' = B \cap B' =\emptyset$.
\item[$\mathbf{R2}$] $ A \cup C = B \cup C~\text{for some}~ C \in \CH $.
\end{enumerate}
\end{remark} 

\begin{proof} 
Clearly, $  \mathbf{R1} \implies \mathbf{R} $ and  $\mathbf{R2} \implies \mathbf{R}$. Suppose that $ A \cup A' = B \cup B'~\text{for some}~ A', B' \in \CH$. Then $ A' \setminus A,~ B' \setminus B, ~A' \cup B' \in \CH$. 
One sees that 
$$ A \cup (A'\cup B') = (A \cup A') \cup B' = (B \cup B')\cup B' = B \cup B', $$
 $$ B \cup (A'\cup B')= (A'\cup B') \cup B =A' \cup (A \cup A')=A \cup A',$$
and hence, $ A \cup (A'\cup B') = B \cup (A'\cup B')$. Thus, $ \mathbf{R} \implies  \mathbf{R2}$. 

Also, we have that $A \cup (A' \setminus A) = B \cup (B' \setminus B)$
with $A \cap (A' \setminus A) = B \cap (B' \setminus B) =\emptyset.$ Thus, $ \mathbf{R} \implies  \mathbf{R1}$. 
\end{proof}

We denote the equivalent class of $A \in \CB$ with respect to $\sim$ by $[A]$ (or $[A]_\CH$ if we need to specify the ideal $\CH$) and the set of all equivalent classes of $\CB$ by $\CB / \CH$. It is easy to check that $\CB / \CH$ is a Boolean algebra with operations defined by 
\[
[A]\cap [B]=[A\cap B], [A]\cup [B]=[A\cup B] ~\text{and}~ [A]\setminus [B]=[A\setminus B].
\]
The partial order $\subseteq$ on $\CB / \CH$ is characterized by  
\begin{align*} 
[A] \subseteq [B] & \iff A \subseteq B\cup W  ~\text{for some}~  W \in \CH \\
&\iff  [A] \cap [B] =[A].
\end{align*}

If in addition $\CH$ is hereditary, and we define $\theta_{\af}([A])=[\theta_{\af}(A)]$ for all $[A] \in \CB/\CH$ and $ \af \in \CL$, then  $(\CB / \CH, \CL,\theta)$ becomes a Boolean dynamical system. We call it a {\em quotient Boolean dynamical system} of $(\CB,\CL,\theta)$.

\vskip 1pc

\section{$C^*$-algebras of generalized boolean dynamical systems}\label{GBDS}
In this section, we introduce a definition of  generalized Boolean dynamical systems and their $C^*$-algebras. 
Let $(\CB,\CL,\theta)$ be a Boolean dynamical system and let 
\[
\mathcal{R}_\alpha^{(\CB,\CL,\theta)}:=\{A\in\mathcal{B}:A\subseteq\theta_\alpha(B)\text{ for some }B\in\mathcal{B}\}
\]
for each $\alpha \in \mathcal{L}$. Note that each $\CR_\af^{(\CB,\CL,\theta)}$ is an ideal of $\CB$. 

We will often, when it is clear which Boolean dynamical system we are working with, just write $\CR_\af$ instead of $\CR_\af^{(\CB,\CL,\theta)}$.

\begin{remark} 
In \cite{COP}, the notaion $\CR_\af$ is used to denote the least-upper bound of $\{\theta_\af(A)\}_{A \in \CB}$ when $\theta_\af$ has compact range. 
\end{remark}

\begin{dfn} \label{def:GBDS} 
A {\em generalized Boolean dynamical system} is a quadruple  $(\CB,\CL,\theta,\CI_\alpha)$ where  $(\CB,\CL,\theta)$ is  a Boolean dynamical system  and  $\{\CI_\alpha:\alpha\in\CL\}$ is a family of ideals in $\CB$ such that $\CR_\alpha\subseteq\CI_\alpha$ for each $\alpha\in\CL$. A {\em  relative generalized Boolean dynamical system} is a pentamerous $(\CB,\CL,\theta,\CI_\alpha;\CJ)$ where  $(\CB,\CL,\theta,\CI_\alpha)$ is a  generalized Boolean dynamical system  and  $\CJ$ is an ideal of  $\CB_{reg}$.
\end{dfn}

\begin{dfn}\label{def:representation of RGBDS} 
Let $(\CB,\CL,\theta, \CI_\af; \CJ)$ be a relative generalized Boolean dynamical system. A {\it  $(\CB, \CL, \theta, \CI_\af;\CJ)$-representation} is a family of projections $\{P_A:A\in\mathcal{B}\}$ and a family of partial isometries $\{S_{\alpha,B}:\alpha\in\mathcal{L},\ B\in\mathcal{I}_\alpha\}$ such that for $A,A'\in\mathcal{B}$, $\alpha,\alpha'\in\mathcal{L}$, $B\in\mathcal{I}_\alpha$ and $B'\in\mathcal{I}_{\alpha'}$,
\begin{enumerate}
\item[(i)] $P_\emptyset=0$, $P_{A\cap A'}=P_AP_{A'}$, and $P_{A\cup A'}=P_A+P_{A'}-P_{A\cap A'}$;
\item[(ii)] $P_AS_{\alpha,B}=S_{\alpha,  B}P_{\theta_\af(A)}$;
\item[(iii)] $S_{\alpha,B}^*S_{\alpha',B'}=\delta_{\alpha,\alpha'}P_{B\cap B'}$;
\item[(iv)] $P_A=\sum_{\af \in\Delta_A}S_{\af,\theta_\af(A)}S_{\af,\theta_\af(A)}^*$ for all  $A\in \mathcal{J}$. 
\end{enumerate}
\end{dfn}

Given a $(\CB, \CL, \theta, \CI_\af;\CJ)$-representation $\{P_A, S_{\af,B}\}$ in a $C^*$-algebra $\CA$, we denote by $C^*(P_A, S_{\af,B})$ the $C^*$-subalgebra of $\CA$ generated by $\{ P_A,  S_{\af,B}\}$.

We will show in Section \hyperref[RGBDS is RCP]{5} that there exists a universal $(\CB, \CL, \theta, \CI_\af;\mathcal{J})$-representation $\{p_A, s_{\af,B}: A\in \CB, \af \in \CL ~\text{and}~ B \in \CI_\af\}$ in the sense that if $\{P_A,S_{\af,B}\}$ is a   $(\CB, \CL, \theta, \CI_\af;\CJ)$-representation in a $C^*$-algebra $\CA$, then there exists a unique $*$-homomorphism $\pi_{S,P}: C^*(p_A,s_{\af,B}) \to \CA$ such that $\pi_{S,P}(p_A)=P_A$ and $\pi_{S,P}(s_{\af,B})=S_{\af,B}$ for $A \in \CB$, $\af \in \CL$ and $B \in \CI_\af$.  
 We write $C^*(\mathcal{B},\mathcal{L},\theta, \CI_\af;\mathcal{J})$ for $C^*(p_A,s_{\af,B})$ and    call it the {\it  $C^*$-algebra of $(\CB,\CL,\theta,\CI_\alpha,\CJ)$}.

By a {\it Cuntz--Krieger representation of $(\CB, \CL,
\theta,\CI_\af)$} we mean a $(\CB, \CL, \theta,\CI_\af;
\CB_{reg})$-representation, and by a {\it  Toeplitz representation of
$(\CB, \CL, \theta,\CI_\af)$} we mean a $(\CB, \CL, \theta,\CI_\af;
\emptyset)$-representation. We write $C^*(\CB,\CL,\theta, \CI_\af)$
for $C^*(\CB, \CL, \theta, \CI_\af;\CB_{reg})$ and call it the {\it
$C^*$-algebra of $(\CB,\CL,\theta,\CI_\alpha)$}, and we write
$\CT(\CB, \CL, \theta, \CI_\af)$ for $C^*(\CB, \CL, \theta,\CI_\af;
\emptyset)$ and call it the {\it Toeplitz $C^*$-algebra of
$(\CB,\CL,\theta,\CI_\alpha)$}.

\vskip1pc

 When $(\CB,\CL,\theta)$ is a Boolean dynamical system, then we write
$C^*(\CB,\CL,\theta)$ for $C^*(\CB,\CL,\theta, \CR_\af)$ and call it
the \emph{$C^*$-algebra of $(\CB,\CL,\theta)$}. If moreover $\CJ$ is
an ideal of $\CB_{reg}$, then we write $C^*(\CB,\CL,\theta;\CJ)$ for
$C^*(\CB,\CL,\theta, \CR_\af;\CJ)$ and call it the \emph{$C^*$-algebra
of $(\CB,\CL,\theta;\CJ)$}.

We shall in Proposition \ref{RGBDS isom GBDS} see that the
$C^*$-algebra $C^*(\CB,\CL,\theta,\CR_\alpha;\CJ)$ of a relative
generalized Boolean dynamical system is isomorphic to the
$C^*$-algebra of a (different) generalized Boolean dynamical system
$(\tilde{\CB},\CL,\tilde{\theta},\tilde{\CR}_\alpha)$. Moreover, it
follows from Proposition \ref{full hereditary} that
$C^*(\tilde{\CB},\CL,\tilde{\theta},\tilde{\CR}_\alpha)$ and
$C^*(\tilde{\CB},\CL,\tilde{\theta})$ are Morita equivalent. We thus
have that the $C^*$-algebra of any relative generalized Boolean
dynamical system is Morita equivalent to the $C^*$-algebra of a
Boolean dynamical system.

\begin{remark}  It follows from the universal property of $C^*(\CB,\CL,\theta,
\CI_\af;\CJ)=C^*(p_A, s_{\af,B})$ that there is a strongly continuous
action $\gm:\mathbb T\to {\rm Aut}(C^*(\CB,\CL,\theta, \CI_\af;\CJ))$,
which we call the {\it gauge action}, such that

\[
\gm_z(p_A)=p_A   \ \text{ and } \ \gm_z(s_{\af,B})=zs_{\af,B}
\]
for $A\in \CB$, $\af \in \CL$ and $B \in \CI_\af$.
\end{remark}

\vskip1pc

\begin{remark} 
Let $(\CB,\CL,\theta,\CI_\af;\CJ)$ be a relative  generalized  Boolean dynamical system. The fact that a family    $\{P_A, S_{\af,B}: A \in \CB,~ \af \in \CL ~\text{and}~ B \in \CI_\af \}$ satisfies (i)-(iii) in Definition \ref{def:representation of RGBDS} is equivalent to that the family satisfies that 
\begin{enumerate}
\item[(a)] $P_\emptyset=0$, $P_{A\cap A'}=P_AP_{A'}$, and $P_{A\cup A'}=P_A+P_{A'}-P_{A\cap A'}$;
\item[(b)] $P_AS_{\alpha,B}=S_{\alpha, \theta_\af(A)\cap B}$;
\item[(c)] $S_{\alpha,B}P_A=S_{\alpha, B \cap A}$;
\item[(d)] $S_{\alpha,B}^*S_{\alpha',B'}=\delta_{\alpha,\alpha'}P_{B\cap B'}$
\end{enumerate}
for $A,A'\in\mathcal{B}$, $\alpha,\alpha'\in\mathcal{L}$, $B\in\mathcal{I}_\alpha$ and $B'\in\mathcal{I}_{\alpha'}$.
\end{remark}

\begin{proof}  
By (b) and (c), we see that $P_AS_{\af,B}=S_{\af, \theta_\af(A) \cap B}=S_{\af, B}P_{\theta_\af(A)}$. 

For the converse, note that $S_{\af,B}^*P_A=(P_AS_{\af, B})^*=(S_{\af,B}P_{\theta_\af(A)})^*=P_{\theta_\af(A)}S_{\af,B}^*$. It then follows that
 \begin{align*} &\|P_AS_{\af,B}-S_{\af, \theta_\af(A) \cap B}\|^2 \\&=\|(P_AS_{\af,B}-S_{\af, \theta_\af(A) \cap B})^*(P_AS_{\af,B}-S_{\af, \theta_\af(A) \cap B})\| \\
&=\|S_{\af,B}^*P_AS_{\af,B}-S_{\af,B}^*P_AS_{\af, \theta_\af(A)\cap B}-S_{\af, \theta_\af(A)\cap B}^*P_AS_{\af,B}+P_{\theta_\af(A)\cap B}\| \\
&=\|S_{\af,B}^*S_{\af,B}P_{\theta_\af(A)}-P_{\theta_\af(A)}S_{\af,B}^*S_{\af, \theta_\af(A)\cap B}-S_{\af, \theta_\af(A)\cap B}^*S_{\af,B}P_{\theta_\af(A)}+P_{\theta_\af(A)\cap B}\| \\
&=\|P_{\theta_\af(A)\cap B}-P_{\theta_\af(A)\cap B}-P_{\theta_\af(A)\cap B}+P_{\theta_\af(A)\cap B}\|=0.
\end{align*}
Thus, we have $P_AS_{\af,B}=S_{\af, \theta_\af(A) \cap B}$. One also  sees that  
$\|S_{\af,B}P_A - S_{\af, B \cap A}\|^2=0$,
and hence, we have $S_{\af,B}P_A = S_{\af, B \cap A}$. 
\end{proof}

For $\af=\af_1\af_2 \cdots \af_n \in \CL^* \setminus \{\emptyset\}$, we define
\[
\CI_\af:=\{A \in \CB : A \subseteq \theta_{\af_2 \cdots \af_n}(B)~\text{for some }~ B \in \CI_{\af_1}\}.
\]
 For $\af =\emptyset$, we  define $\CI_\emptyset := \CB$

\begin{dfn} Let $\{P_A,\ S_{\alpha,B}: A\in\CB,\ \alpha\in\CL,\ B\in\CI_\alpha\}$ be a $(\CB,\CL,\theta,\CI_\alpha;\CJ)$-representation.
For $\af=\af_1\af_2 \cdots \af_n \in \CL^* \setminus \{\emptyset\}$ and $A \in \CI_{\af}$, we define
\[
S_{\af,A}:=S_{\af_1,B}S_{\af_2, \theta_{\af_2}(B)}S_{\af_3, \theta_{\af_2\af_3}(B)} \cdots S_{\af_n,A},
\] 
where $B \in \CI_{\af_1}$ is such that $A \subseteq \theta_{\af_2 \cdots \af_n}(B)$. For $\af = \emptyset$, we also define $S_{\emptyset, A}:=P_A$.
\end{dfn}

\begin{remark} 
In the above definition, $S_{\af,A}$ is independent of the choice of $B \in \CI_{\af_1}$. It is enough to show that for $|\af|=2$. Put $\af=\af_1\af_2 \in \CL^*$,  $A \subseteq \theta_{\af_2}(B_1)$ and $A \subseteq \theta_{\af_2}(B_2)$ for some $B_1, B_2 \in \CI_{\af_1}$. Then we see that 
\begin{align*} 
&\|S_{\af_1, B_1}S_{\af_2, A}-S_{\af_1, B_2}S_{\af_2, A}\|^2\\
&=\|(S_{\af_2,A}^*S_{\af_1, B_1}^*-S_{\af_2,A}^*S_{\af_1, B_2}^*)(S_{\af_1, B_1}S_{\af_2, A}-S_{\af_1,B_2}S_{\af_2, A})\|\\
&=\|S_{\af_2,A}^*P_{B_1}S_{\af_2,A}-S_{\af_2,A}^*P_{B_1 \cap B_2}S_{\af_2,A} \\& \hskip 9pc- S_{\af_2,A}^*P_{B_1 \cap B_2}S_{\af_2,A}+ S_{\af_2,A}^*P_{B_2}S_{\af_2,A}\| \\
&=\| P_{A \cap \theta_{\af_2}(B_1)}-P_{A \cap \theta_{\af_2}(B_1 \cap B_2)}-P_{A \cap \theta_{\af_2}(B_1 \cap B_2)}+P_{A \cap \theta_{\af_2}(B_2)}\| \\
&=\|P_A -P_A-P_A+P_A\|=0. 
\end{align*}
Thus, it follows that $S_{\af_1, B_1}S_{\af_2, A}=S_{\af_1, B_2}S_{\af_2, A}$.
\end{remark}

It is straightforward to check the following lemma.

\begin{lem}  Let $\{P_A,\ S_{\alpha,B}: A\in\CB,\ \alpha\in\CL,\ B\in\CI_\alpha\}$ be a $(\CB,\CL,\theta,\CI_\alpha;\CJ)$-representation. For
 $\af \in \CL^*$, $A \in \CB$ and $B \in \CI_\af$, we have  
\begin{enumerate}
\item $P_AS_{\af,B}=S_{\af,B} P_{\theta_\af(A)}$,
\item $P_AS_{\alpha,B}=S_{\alpha, \theta_\af(A)\cap B}$,
\item $S_{\alpha,B}P_A=S_{\alpha, B \cap A}$.
\end{enumerate}
\end{lem}

\begin{lem}\label{mul-s*s} 
 Let $\{P_A,\ S_{\alpha,B}: A\in\CB,\ \alpha\in\CL,\ B\in\CI_\alpha\}$ be a $(\CB,\CL,\theta,\CI_\alpha;\CJ)$-representation. For $\af, \bt \in \CL^*$, $A \in \CI_\af$ and $B \in \CI_\bt$, we have
\[
S_{\af,A}^*S_{\bt,B}= \left\{ 
\begin{array}{ll}
    P_{A \cap B} & \hbox{if\ }\af =\bt \\
    S_{\af', A \cap \theta_{\af'}(B)}^* & \hbox{if\ }\af =\bt\af' \\
    S_{\bt',B \cap \theta_{\bt'}(A)}   & \hbox{if\ } \bt=\af\bt' \\
    0 & \hbox{otherwise.} \\
\end{array}
\right.
\]
\end{lem}

\begin{proof} 
Let $\af=\af_1\af_2 \cdots \af_n$, $\bt=\bt_1\bt_2 \cdots \bt_m  \in \CL^*$ and put 
\begin{align*}
S_{\af, A} &=S_{\af_1,A'}S_{\af_2,\theta_{\af_2}(A')}S_{\af_3, \theta_{\af_2\af_3}(A')} \cdots S_{\af_n,A}, \\
S_{\bt, B}&=S_{\bt_1,B'}S_{\bt_2,\theta_{\bt_2}(B')}S_{\bt_3, \theta_{\bt_2\bt_3}(B')} \cdots S_{\bt_m,B},
\end{align*}
where $A' \in \CI_{\af_1}$ such that $A \subseteq \theta_{\af_2 \cdots \af_n}(A')$ and $B' \in \CI_{\bt_1}$ such that $B \subseteq \theta_{\bt_2 \cdots \bt_m}(B')$. If $\af=\bt$, we have 
\begin{align*}
S_{\af,A}^*S_{\af,B} &= ( S_{\af_n,A}^* \cdots S_{\af_2,\theta_{\af_2}(A')}^* S_{\af_1, A'}^*)(S_{\af_1,B'}S_{\af_2,\theta_{\af_2}(B')} \cdots S_{\af_n,B})\\
&=S_{\af_n,A}^* \cdots S_{\af_2,\theta_{\af_2}(A')}^*P_{A' \cap B'}S_{\af_2,\theta_{\af_2}(B')} \cdots S_{\af_n,B} \\
&= S_{\af_n,A}^* \cdots S_{\af_2,\theta_{\af_2}(A')}^*S_{\af_2,\theta_{\af_2}(B')} P_{\theta_{\af_2}(A' \cap B')} \cdots S_{\af_n,B} \\
&= S_{\af_n,A}^* \cdots P_{\theta_{\af_2}(A' \cap B')} \cdots S_{\af_n,B} \\
& \hskip 0.5pc  \vdots \\
&= S_{\af_n,A}^* S_{\af_n,B}  P_{\theta_{\af_2 \cdots \af_n}(A' \cap B')}  \\
&=P_{A \cap B}. 
\end{align*}
If $\af =\bt\af'$, we first note that 
\begin{align*}
S_{\bt\af',A} &=\big(S_{\bt_1, A'}S_{\bt_2, \theta_{\bt_2}(A')} \cdots S_{\bt_{|\bt|}, \theta_{\bt_2 \cdots \bt_{|\bt|}}(A')} \big)S_{\af_1', \theta_{\bt_2 \cdots \bt_{|\bt|}\af_1'}(A')} \cdots S_{\af'_{|\af'|}, A} \\&=S_{\bt, \theta_{\bt_2 \cdots \bt_{|\bt|}}(A')}S_{\af',A}.
\end{align*}
Thus, we have
\begin{align*} 
S_{\bt\af', A}^*S_{\bt,B} &=  (S_{\bt, \theta_{\bt_2 \cdots \bt_{|\bt|}}(A')}S_{\af',A})^*S_{\bt, B} \\
&=S_{\af',A}^*S_{\bt, \theta_{\bt_2 \cdots \bt_{|\bt|}}(A')}^*S_{\bt, B} \\
&=S_{\af',A}^*P_{(\theta_{\bt_2 \cdots \bt_{|\bt|}}(A'))\cap B} \\
&= S_{\af', (\theta_{\bt_2 \cdots \bt_{|\bt|}\af'}(A')) \cap \theta_{\af'}(B) \cap A}^* \\
&=S_{\af', A \cap \theta_{\af'}(B) }^*. 
\end{align*}
Similarly, we also have $S_{\af,A}^*S_{\af\bt', B}=S_{\bt',B \cap \theta_{\bt'}(A)}$. Otherwise, $S_{\af,A}^*S_{\bt,B}=0$ by Definition~\ref{def:representation of RGBDS}(iii). 
\end{proof}

As a corollary of Lemma \ref{mul-s*s}, we have the following. 

\begin{lem}\label{cal}
 Let $\{P_A,\ S_{\alpha,B}: A\in\CB,\ \alpha\in\CL,\ B\in\CI_\alpha\}$ be a $(\CB,\CL,\theta,\CI_\alpha;\CJ)$-representation. For $\af,\bt,\mu,\nu \in \CL^*$, $A \in \CI_\af$, $B \in \CI_\bt$,  $C \in \CI_\mu$ and $D \in \CI_\nu$,   we have
\[
(S_{\af,A}S_{\bt,B}^*)(S_{\mu,C}S_{\nu,D}^*)= \left\{ 
\begin{array}{ll}
    S_{\af,A \cap B \cap C}S_{\nu, D \cap B \cap C}^* & \hbox{if\ }\bt =\mu \\
    S_{\af,A}S_{\nu\bt', \theta_{\bt'}(C \cap D) \cap B}^*    & \hbox{if\ }\bt =\mu\bt' \\
    S_{\af\mu',\theta_{\mu'}(A \cap B ) \cap C}S_{\nu,D}^*                  & \hbox{if\ } \mu=\bt\mu' \\
    0   & \hbox{otherwise.} \\
\end{array}
\right.
\]
\end{lem}

\begin{remark}
Let $\{P_A,\ S_{\alpha,B}: A\in\CB,\ \alpha\in\CL,\ B\in\CI_\alpha\}$
be a $(\CB,\CL,\theta,\CI_\alpha;\CJ)$-representation. We then have
that
\begin{align}
C^*(P_A,S_{\alpha,B})&=\overline{\operatorname{span}}\{
S_{\af,A}S_{\bt,B}^*: ~\af,\bt \in \CL^* ~\text{and}~ A \in \CI_\af, B
\in \CI_\bt\}\label{eq:2}\\
&=\overline{{\rm \operatorname{span}}}\{S_{\af,A}S_{\bt,A}^*: \af,\bt
\in \CL^* ~\text{and}~ A \in \CI_\af\cap \CI_\bt \}.\label{eq:3}
\end{align}
\end{remark}

\begin{proof}
It follows from Lemma \ref{cal} that the right-hand side of \eqref{eq:2} is
a $C^*$-subalgebra of $C^*(P_A,S_{\alpha,B})$, and since it contains
all the generators of $C^*(P_A,S_{\alpha,B})$, they must be equal. To
see the last equality, we only need to observe that
$S_{\af,A}S_{\bt,B}^*=S_{\af,A}P_{A\cap B}S_{\bt,B}^*=S_{\af,A \cap
B}S_{\bt,A \cap B }^*$  for $\af,\bt \in \CL^* $ and $ A \in \CI_\af,
B \in \CI_\bt.$
\end{proof}

\section{Examples}\label{examples}
The class of $C^*$-algebras of Boolean dynamical system described in Example \ref{BDS of cr-cd} below  was introduced in \cite{COP} to study labelled graph $C^*$-algebras from a more general point of view. 
 It is shown in \cite{COP} that the class of $C^*$-algebras of Boolean dynamical system contains  weakly left-resolving normal labelled graph $C^*$-algebras under some assumption.   We show in the following two examples that our $C^*$-algebras of generalized Boolean dynamical systems contains all  $C^*$-algebras of Boolean dynamical systems  and all weakly left-resolving normal labelled graph $C^*$-algebras.

\begin{ex}\label{BDS of cr-cd} 
Let $(\CB,\CL,\theta)$ be a Boolean dynamical system  with assumptions that $\theta_\af$ has compact range $\CR_{\theta_\af}$ and closed domain $\CD_{\theta_\af}$  for each $\af \in \CL$ (see \cite[Definition 3.3]{COP} or Remark \ref{compact range}). We denote the universal Cuntz--Krieger  Boolean $C^*$-algebra constructed in \cite{COP} by $C^*(s_\af, p_A)$. We then show that 
\[
C^*(\CB, \CL, \theta) \cong C^*(s_\af, p_A).
\]

We  check that
\[
\{p_A, s_\af p_B: A \in \CB, \af \in \CL ~\text{and}~ B \in \CR_\af \}	
\]
is a Cuntz-Krieger representation of $(\CB, \CL, \theta)$ in $C^*(s_\af, p_A)$:
\begin{enumerate}
\item[(i)] $p_A$'s clearly satisfy Definition \ref{def:representation of RGBDS}(i),
\item[(ii)] $p_A(s_\af p_B)=s_\af p_{\theta_\af(A)} p_B =(s_\af p_B) p_{\theta_\af(A)}$,
\item[(iii)] $(s_\af p_B)^*(s_{\af'}p_{B'})= p_Bs_\af^*s_{\af'}p_{B'}=\delta_{\af,\af'}p_Bp_{\CR_{\theta_\af}}p_{B'}=\delta_{\af,\af'}p_{B \cap B'}$,
\item[(iv)] $p_A= \sum_{\af \in \Delta_A} s_\af p_{\theta_\af(A)}s_\af^*=\sum_{\af \in \Delta_A} (s_\af p_{\theta_\af(A)})(s_\af p_{\theta_\af(A)})^*$ for all $A \in \CB_{reg}$.
\end{enumerate}
Then the universal property of $C^*(\CB, \CL, \theta)$ gives  a $*$-homomorphism 
\[
\phi: C^*(\CB, \CL, \theta) \to C^*(s_\af, p_A)	
\]
defined by 
\[
\phi(p_A)=p_A ~\text{and}~ \phi(s_{\af,B})=s_\af p_B	
\]
for all $A \in \CB, \af \in \CL$ and $B \in \CR_\af$. Since $\theta_\af(\CD_{\theta_\af})=\CR_{\theta_\af}$ for each $\af \in \CL$, we have $\CR_{\theta_\af} \in \CR_\af$ for each $\af \in \CL$. Thus it follows from the equation $s_\af=s_\af p_{\CR_{\theta_\af}}$  that the family $\{p_A, s_\af p_B : A \in \CB, \af \in \CL ~\text{and}~ B \in \CR_\af\}$ contains all generators of $C^*(s_\af,p_A)$. Thus $\phi$ is surjective.  Since $p_A \neq 0$ for all $A  \neq \emptyset$ and  $C^*(s_\af,p_A)$ admits a gauge action $\bt: \mathbb{T} \to \text{Aut}(C^*(s_\af,p_A))$ such that $\bt_z(p_A)=p_A$ and $\bt_z(s_\af p_B)=zs_\af p_B$ for every $A \in \CB$, $\af \in \CL$  and $B \in \CR_\af$, the gauge-invariant uniqueness theorem \ref{GIUT of GBDS} implies that $\phi$ is an isomorphism.
\end{ex}

\begin{ex}[Weakly left-resolving normal labelled spaces]\label{labelled graph} 
We refer the reader to \cite{BaCaPa2017, BP1,BP2, JKP} for the basic definitions and terminology of labelled graphs and their $C^*$-algebras. Let $(E, \CL)$ be a labelled graph over  $\CA$ consisting of a directed graph $E$ and a labelling map $\CL: E^1 \to \CA$ which is assumed to be onto.  We then consider  a weakly left-resolving normal labelled space  $(E,\CL,\CB)$ and put $C^*(E,\CL,\CB):=C^*(p_A, s_\af)$. Then $\CB$ is a Boolean algebra and for each $\af \in \CA$, the map $\theta_{\af}:\CB \to \CB$ defined by \begin{align}\label{labelled graph:action}\theta_\af(A):=r(A,\af)\end{align} is an action on $\CB$ (see \cite[Example 11.1]{COP}). Then the triple $(\CB, \CA, \theta)$ is a Boolean dynamical system. It is clear that $\CR_\af \subseteq \CI_{r(\af)} (=\{A \in \CB : A \subseteq r(\af)\})$  for each $\af \in \CL$. We claim that 
\[
C^*(E,\CL,\CB) \cong C^*(\CB, \CA,\theta, \CI_{r(\af)}).
\]
It it straightforward to check that 
\[
\{p_A, s_\af p_B: A \in \CB, \af \in \CA ~\text{and}~ B \in \mathcal{I}_{r(\alpha)}\}	
\]
is a Cuntz--Krieger representation of $(\CB, \CA,\theta,\CI_{r(\af)} )$. Then the universal property of $C^*(\CB, \CA,\theta, \CI_{r(\af)})$ gives a $*$-homomorphism 
\[
\phi: C^*(\CB, \CA,\theta, \CI_{r(\af)}) \to C^*(E,\CL,\CB)	
\]
defined by 
\[
\phi(p_A)=p_A ~\text{and}~ \phi(s_{\af,B})=s_\af p_B	
\]
for all $A \in \CB, \af \in \CA$ and $B \in \mathcal{I}_{r(\alpha)}$. Since $s_\af=s_\af p_{r(\af)}$, the family $\{p_A, s_\af p_B: A \in \CB, \af \in \CA ~\text{and}~ B \in \mathcal{I}_{r(\alpha)} \}$ generates $C^*(E, \CL, \CB)$, and hence, the map $\phi$ is onto. Applying the gauge-invariant uniqueness theorem \ref{GIUT of GBDS}, we conclude that  $\phi$ is an isomorphism.
\end{ex}

\begin{remark} 
We continue Example \ref{labelled graph}. We claim that 
\begin{enumerate}
\item If, in addition, $E$ has no source and $(E,\CL,\CB)$ is reciever set-finite, then we have $C^*(E, \CL, \CB) \cong C^*(\CB, \CA, \theta).$
\item In general, $C^*(E, \CL, \CB) $ is not isomorphic to $C^*(\CB, \CA, \theta).$
\item $C^*(E,\CL,\CB)$ is Morita equivalent to $ C^*(\CB, \CA,\theta) $.
\end{enumerate} 
\end{remark}

\begin{proof} 
(1): For each $\af \in \CA$, put
$\CR_\af=\{B \in \CB: B \subseteq r(A, \af) ~\text{for some}~ A \in \CB \}$. Then 
\[
\{p_A, s_\af p_B: A \in \CB, \af \in \CA ~\text{and}~ B \in \CR_\af \}
\]
is a Cuntz--Krieger representation of $(\CB, \CA, \theta)$. Clearly, $p_A$'s satisfy Definition \ref{def:representation of RGBDS} (i). We show that  the family satisfies Definition \ref{def:representation of RGBDS}(ii)-(iv):
\begin{enumerate}
\item[(ii)] $p_A(s_\af p_B)=s_\af p_{r(A,\af)} p_B =(s_\af p_B) p_{r(A,\af)}$,
\item[(iii)] $(s_\af p_B)^*(s_{\af'}p_{B'})= p_Bs_\af^*s_{\af'}p_{B'}=\delta_{\af,\af'}p_Bp_{r(\af)}p_{B'}=\delta_{\af,\af'}p_{B \cap B'}$,
\item[(iv)] $p_A= \sum_{\af \in  \CL(AE^1)} s_\af p_{r(A,\af)}s_\af^*=\sum_{\af \in  \Delta_A} (s_\af p_{r(A,\af)})(s_\af p_{r(A,\af)})^*$ for all $A \in \CB_{reg}$.
\end{enumerate}
Then the universal property of $C^*(\CB, \CA, \theta)$ gives  a $*$-homomorphism 
\[
\phi: C^*(\CB, \CA, \theta) \to C^*(p_A, s_\af p_B)	
\]
defined by 
\[
\phi(p_A)=p_A ~\text{and}~ \phi(s_{\af,B})=s_\af p_B	
\]
for all $A \in \CB, \af \in \CA$ and $B \in \CR_\af$. Since $p_A \neq 0$ for all $A \neq \emptyset$ and  $C^*(p_A, s_\af p_B)$ admits a gauge action $\bt: \mathbb{T} \to \text{Aut}(C^*(E, \CL, \CB))$ such that $\bt_z(p_A)=p_A$ and $\bt_z(s_\af p_B)=zs_\af p_B$ for every $A \in \CB$, $\af \in \CA$  and $B \in \CR_\af$, the gauge-invariant uniqueness theorems \ref{GIUT of GBDS} implies that $\phi$ is an isomorphism.

If $E$ has no sources and $(E, \CL,\CB)$ is reciever set-finite, then we have 
\[
r(\af)=r(s(\af), \af)= r(\cup_{i=1}^m [w_i]_1, \af)= \cup_{i=1}^m r([w_i]_1, \af)
\] 
for some $w_i \in s(\af)$ $(i=1, \cdots, m)$. The union is disjoint since $(E, \CL,\CB)$ is weakly left-resolving. We then see that 
\[
s_\af=s_\af p_{r(\af)}=\sum_{i=1}^m s_\af p_{r([w_i]_1, \af)}
\] 
for some $r([w_i]_1, \af) \in \CR_\af$. Thus, the family  $\{p_A, s_\af p_B: A \in \CB, \af \in \CA ~\text{and}~ B \in \CR_\af \} $ generates $C^*(E, \CL, \CB)$, and hence, the map $\phi$ becomes a $*$-homomorphism from $C^*(\CB, \CA, \theta)$ onto  $C^*(E, \CL, \CB)$. The gauge-invariant uniqueness theorems \ref{GIUT of GBDS} implies that $\phi$ is an isomorphism.

\vskip 0.5pc

(2): Let $E^0:=\mathbb{N}\times\{1,2\}$, $E^1:=\mathbb{N}$ and define $r,s:E^1\to E^0$  by $s(e)=(e,1)$, $r(e)=(e,2)$. Then $E:=(E^0,E^1,r,s)$ is the following directed graph.

\vskip 1.5pc 
\hskip 3cm
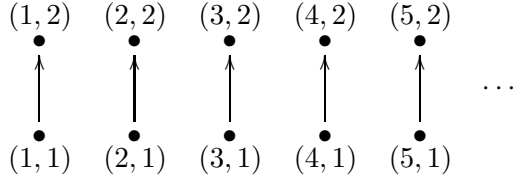
\xy  /r0.3pc/:(48.3,5)*+{\cdots };
 (0,0)*+{\bullet}="V0";
(10,0)*+{\bullet}="V1"; 
(20,0)*+{\bullet}="V2";
(30,0)*+{\bullet}="V3";
(40,0)*+{\bullet}="V4";
(0,10)*+{\bullet}="W0";
(10,10)*+{\bullet}="W1"; 
(20,10)*+{\bullet}="W2";
(30,10)*+{\bullet}="W3";
(40,10)*+{\bullet}="W4";
 "V0";"W0"**\crv{(0,0)&(0,10)}; ?>*\dir{>}\POS?(.5)*+!D{};
"V1";"W1"**\crv{(10,0)&(10,10)}; ?>*\dir{>}\POS?(.5)*+!D{};
"V2";"W2"**\crv{(20,0)&(20,10)}; ?>*\dir{>}\POS?(.5)*+!D{};
"V3";"W3"**\crv{(30,0)&(30,10)}; ?>*\dir{>}\POS?(.5)*+!D{};
"V4";"W4"**\crv{(40,0)&(40,10)}; ?>*\dir{>}\POS?(.5)*+!D{};
  (0.1,-2.5)*+{(1,1)};(10.1,-2.5)*+{(2,1)};
(20.1,-2.5)*+{(3,1)};(30.1,-2.5)*+{(4,1)};(40.1,-2.5)*+{(5,1)};
 (0.1,12.5)*+{(1,2)}; (10.1,12.5)*+{(2,2)};
 (20.1,12.5)*+{(3,2)}; (30.1,12.5)*+{(4,2)}; (40.1,12.5)*+{(5,2)};
 \endxy 
\vskip 1.5pc 

\noindent
Put $\mathcal{A}:=\{a\}$ and define $\mathcal{L}:E^1\to\mathcal{A}$
by $\mathcal{L}(e)=a$. Then we have the following labelled graph $(E, \CL)$. 

\vskip 1.5pc 
\hskip 3.4cm
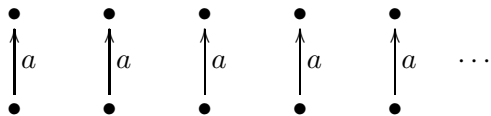
\xy  /r0.3pc/:(48.3,5)*+{\cdots };
 (0,0)*+{\bullet}="V0";
(10,0)*+{\bullet}="V1"; 
(20,0)*+{\bullet}="V2";
(30,0)*+{\bullet}="V3";
(40,0)*+{\bullet}="V4";
(0,10)*+{\bullet}="W0";
(10,10)*+{\bullet}="W1"; 
(20,10)*+{\bullet}="W2";
(30,10)*+{\bullet}="W3";
(40,10)*+{\bullet}="W4";
 "V0";"W0"**\crv{(0,0)&(0,10)}; ?>*\dir{>}\POS?(.5)*+!D{};
"V1";"W1"**\crv{(10,0)&(10,10)}; ?>*\dir{>}\POS?(.5)*+!D{};
"V2";"W2"**\crv{(20,0)&(20,10)}; ?>*\dir{>}\POS?(.5)*+!D{};
"V3";"W3"**\crv{(30,0)&(30,10)}; ?>*\dir{>}\POS?(.5)*+!D{};
"V4";"W4"**\crv{(40,0)&(40,10)}; ?>*\dir{>}\POS?(.5)*+!D{};
(1.5,5)*+{a};(11.5,5)*+{a};(21.5,5)*+{a};(31.5,5)*+{a};(41.5,5)*+{a};
  \endxy 
\vskip 1.5pc 

\noindent
For $A,B\subseteq\mathbb{N}$, we write $(A,B)$ for the subset $A\times\{1\}\cup B\times\{2\}$ of $E^0$. Let $$\mathcal{B}:=\{(A,B):A\text{ is a finite subset of }\mathbb{N},\ B\text{ is a subset of }\mathbb{N}\text{ such that }B\text{ or }\mathbb{N}\setminus B\text{ is finite}\}.$$ Then  $(E,\mathcal{L},\mathcal{B})$ is a weakly left-resolving normal labelled space. Define $\theta_a:\mathcal{B}\to\mathcal{B}$ by $\theta_a((A,B)):=r((A,B),a)=(\emptyset,A)$. Then $(\mathcal{B},\mathcal{A},\theta)$ is a Boolean dynamical system such that $\mathcal{R}_a=\{(\emptyset,A):A\text{ is a finite subset of }\mathbb{N}\}$.

Let $\widehat{\mathbb{N}}$ be the one-point compactification of $\mathbb{N}$, and let $A$ be the $C^*$-algebra of $2\times 2$ matrices over $C(\widehat{\mathbb{N}})$. For $(A,B)\in\mathcal{B}$, let $$p_{(A,B)}= \begin{pmatrix}1_A&0\\0&1_B\end{pmatrix}.$$ For a finite subset $C$ of $\mathbb{N}$, let $$s_{a,(\emptyset,C)}=\begin{pmatrix}0&1_C\\0&0\end{pmatrix},$$ and let $$s_a= \begin{pmatrix}0&1\\0&0\end{pmatrix}.$$ Then it is straightforward to check that $$\{p_{(A,B)}:(A,B)\in\mathcal{B}\}\cup\{s_a\}$$ is a representation of $(E,\mathcal{L},\mathcal{B})$ and $$\{p_{(A,B)}:(A,B)\in\mathcal{B}\}\cup\{s_{a,(\emptyset,C)}:C\text{ is a finite subset of }\mathbb{N}\}$$ is a Cuntz-Krieger representation of $(\mathcal{B},\mathcal{A},\theta)$. The $C^*$-algebra generated by $$\{p_{(A,B)}:(A,B)\in\mathcal{B}\}\cup\{s_a\}$$ is $A$, while the $C^*$-algebra generated by $$\{p_{(A,B)}:(A,B)\in\mathcal{B}\}\cup\{s_{a,(\emptyset,C)}:C\text{ is a finite subset of }\mathbb{N}\}$$ is $$\left\{\begin{pmatrix}f_{11}&f_{12}\\f_{21}&f_{22}\end{pmatrix}\in A: f_{11},f_{12},f_{21}\in C_0(\mathbb{N})\right\}$$ (where we consider $C_0(\mathbb{N})$ to be an ideal of $C(\widehat{\mathbb{N}})$). This shows that $C^*(\mathcal{B},\mathcal{A},\theta)$ is not isomorphic to $C^*(E,\mathcal{L},\mathcal{B})$.

\vskip 0.5pc

(3): By Example \ref{labelled graph} and Proposition \ref{full hereditary}, we deduce that $C^*(E,\CL,\CB)$ is Morita equivalent to $ C^*(\CB, \CA,\theta) $.
\end{proof}

\section{$C^*$-algebras of  relative generalized Boolean dynamical systems are relative Cuntz-Pimsner algebras}\label{RGBDS is RCP}

In this section, we shall prove that for each relative generalized Boolean dynamical system  $(\CB,\CL,\theta,\CI_\alpha;\CJ)$,  
 there exists a unique universal $(\CB,\CL, \theta,\CI_\af; \CJ)$-representation by realizing  $C^*(\CB,\CL, \theta, \CI_\af;\CJ)$ as a relative Cuntz-Pimsner algebra.
We do that by closely imitating the approach used in \cite[Section
3]{BaCaPa2017}.

We first construct a $C^*$-correspondence of a generalized Boolean dynamical system  $(\CB,\CL,\theta,\CI_\alpha)$. Most of the results can be obtained by arguments similar to arguments used in \cite{BaCaPa2017}. For completenss, we provide detailed proof for the results which we need to modify compared to \cite{BaCaPa2017}.

\vskip 1pc

Let $(\CB, \CL,\theta)$ be a Boolean dynamical system. For each $A \in \CB$, we define 
\[
\chi_A:=1_{Z(A)} \in C_0(\widehat{\CB}).
\]
  Let $\CA(\CB,\CL,\theta)$ denote the $C^*$-subalgebra of $C_0(\widehat{\CB})$ generated by $\{\chi_A: A \in \CB\}$. We then have that   
\[
\CA(\CB,\CL,\theta)=\overline{\operatorname{\operatorname{span}}}\{\chi_A: A \in \CB\}.
\]

The following lemma will be frequently used throughout this section. The lemma can be proved in a way similar to have \cite[Lemmas 3.2 and 3.3]{BaCaPa2017} is proved. We therefore omit the proof.

\begin{lem}\label{prop: A(B,L)} Let $(\CB,\CL,\theta)$ be a Boolean dynamical system. 
\begin{enumerate}
\item Let $\{p_A: A \in \CB\}$ be a family of projections in a $C^*$-algebra $\CX$ such that $p_{A \cap B}=p_A p_B$, $p_{A \cup B}=p_A+p_B-p_{A \cap B}$ for all $A,B \in \CB$ and $p_{\emptyset}=0$.  Then there is a unique $*$-homomorphism $\phi: \CA(\CB,\CL,\theta) \to \CX$ such that $\phi(\chi_A)=p_A$ for all $A \in \CB$. 
\item If $\CI$ is an ideal of $\CA(\CB,\CL,\theta)$, then we have 
$\CI=\overline{\operatorname{span}}\{\chi_A: A \in \CB, \chi_A \in \CI\}.$
\end{enumerate}
\end{lem}

Let $\{\mathcal{I}_\alpha:\alpha\in\mathcal{L}\}$ be  a family of ideals of $\mathcal{B}$ such that $\mathcal{R}_\alpha\subseteq \mathcal{I}_\alpha$ for all $\alpha\in\mathcal{L}$. For each $\af \in \CL$, we let $X_\af$ denote the ideal of $\CA(\CB,\CL,\theta)$ generated by $\{\chi_A: A \in \CI_\af\}$, that is, 
\[
X_\alpha=\overline{\operatorname{\operatorname{span}}}\{\chi_A: A \in \CI_\af\}.
\] 
It is straightforward to see that $X_\af$ is a right Hilbert $\CA(\CB,\CL,\theta)$-module with right action given by the usual multiplication in $\CA(\CB,\CL,\theta)$ and inner product defined by $\left\langle x,y \right\rangle=x^*y$  for $x,y \in X_\alpha$. 

We define 
\begin{align*}
X(\CB,\CL,\theta, \CI_\af)&:=\bigoplus_{\alpha\in\mathcal{L}}X_\alpha\\
&=\biggl\{(x_\af)_{\af \in \CL} \in \prod_{\af \in \CL} X_\af: \sum_{\af \in \CL}x_\af^* x_\af ~\text{converges in}~\CA(\CB,\CL,\theta)\biggr\} 
\end{align*} 
to be the right Hilbert $\CA(\CB,\CL,\theta)$-module with inner product defined by 
\[\bigl\langle (x_\af)_{\af \in \CL}, (y_\af)_{\af \in \CL}\bigr\rangle:=\sum_{\af \in \CL}\bigl\langle x_\af,y_\af\bigr\rangle =\sum_{\af \in \CL} x_\af^*y_\af
\] 
and the right action given by $(x_\af)_{\af \in \CL} \cdot f :=(x_\af f)_{\af \in \CL}$ for $(x_\af)_{\af \in \CL}, (y_\af)_{\af \in \CL} \in X(\CB,\CL,\theta, \CI_\af)$ and $f \in \CA(\CB,\CL,\theta)$.  
It is easy to see that 
\begin{align*} 
X(\CB, \CL,\theta, \CI_\af)&=\overline{\operatorname{span}}\{e_{\af,A}: \af \in \CL ~\text{and}~A \in \CI_\af \}\\
&= \overline{\operatorname{\operatorname{span}}}\{e_{\af, A}: \af \in \CL ~\text{and}~A \in F ~\text{for some}~ F \in \CF\},
\end{align*}
where $e_{\af,A}:=(\delta_{\af,\bt}\chi_{A})_{\bt \in \CL} \in \oplus_{\af \in \CL} X_\af$ and $\CF$ is the collection of finite subsets $F$ of mutually disjoint elements in $\CI_\af$. 

We then have a unique $*$-homomorphism $\phi_\af: \CA(\CB,\CL,\theta) \to X_\af$ defined by $$\phi_\af(\chi_A)=\chi_{\theta_\af(A)}$$ for every $A \in \CB$ and for each $\af \in \CL$ (see \cite[Lemma 3.4]{BaCaPa2017}). If we define a map 
\[
\phi:\CA(\CB,\CL,\theta) \to \mathfrak{L}(X(\CB,\CL,\theta, \CI_\af))
\] 
by 
\[
\phi(f)((x_\alpha)_{\alpha\in\mathcal{L}})=(\phi_\alpha(f)x_\alpha )_{\alpha\in\mathcal{L}}
\]
for $f \in \CA(\CB,\CL,\theta)$ and $(x_\af) \in X(\CB,\CL,\theta, \CI_\af)$,   where $ \mathfrak{L}(X(\CB,\CL,\theta, \CI_\af))$ denotes the $C^*$-algebra of adjointable operators on $X(\CB,\CL,\theta, \CI_\af)$, then the map $\phi$ is a $*$-homomorphism (\cite[Lemma 3.6]{BaCaPa2017}). Thus, $X(\CB,\CL,\theta, \CI_\af)$ is a $C^*$-correspondence over $\CA(\CB,\CL,\theta)$.

\subsection{Relative Cuntz-Pimsner algebras of generalized Boolean dynamical systems}

Let $\mathfrak{K}(X(\CB,\CL,\theta, \CI_\af))$ denote the ideal of $\mathfrak{L}(X(\CB,\CL,\theta, \CI_\af))$ consisting of generalized compact operators, that is, 
\begin{align}\label{compact} 
\mathfrak{K}(X(\CB,\CL,\theta, \CI_\af))=\overline{\operatorname{\operatorname{span}}}\{\Theta_{x,y} \in\mathfrak{L}(X(\CB,\CL,\theta, \CI_\af)):x,y \in X(\CB,\CL,\theta, \CI_\af) \},
\end{align}
where $\Theta_{x,y}$ is the operator that maps $z$ to $x\langle y,z\rangle$. As in \cite[Definition 1.1]{FMR2003}, we define an ideal $J(X(\CB,\CL,\theta, \CI_\af))$ of $\CA(\CB,\CL,\theta)$ by
\[
J(X(\CB,\CL,\theta, \CI_\af)):=\phi^{-1}(\mathfrak{K}(X(\CB,\CL,\theta, \CI_\af))).
\]
We also define an ideal $J_{X(\CB,\CL,\theta, \CI_\af)}$ of $\CA(\CB,\CL,\theta)$ by 
\begin{align*}
J_{X(\CB,\CL,\theta, \CI_\af)}&:= J(X(\CB,\CL,\theta, \CI_\af))\cap \ker(\phi)^{\perp} \\
&=\phi^{-1}(\mathfrak{K}(X(\CB,\CL,\theta, \CI_\af))) \cap \ker(\phi)^{\perp}
\end{align*}
(\cite[Definition 3.2]{Ka2004c}).

We first characterize the ideals $J(X(\CB,\CL,\theta, \CI_\af)) $ and  
$J_{X(\CB,\CL,\theta, \CI_\af)}$ as follows.

\begin{lem}\label{Jx}  
Let $X(\CB,\CL,\theta, \CI_\af)$ be the $C^*$-correspondence over $\CA(\CB,\CL,\theta)$.
\begin{enumerate}
\item $J(X(\CB,\CL,\theta, \CI_\af))
=\overline{\operatorname{\operatorname{span}}}\{\chi_A: A \in \CB ~\text{such that}~\ld_A < \infty \}.$
\item $J_{X(\CB,\CL,\theta, \CI_\af)}=\overline{\operatorname{\operatorname{span}}}\{\chi_A: A \in \CB_{reg}\}.$
\end{enumerate}
\end{lem}

\begin{proof} 
(1): By Lemma \ref{prop: A(B,L)}(2), we only need to show that for $A \in \CB$,  
\begin{align*}
\phi(\chi_A) \in \mathfrak{K}(X(\CB,\CL,\theta, \CI_\af)) & \iff \ld_A < \infty.
\end{align*}

For the "if" direction, choose $A \in \CB$ such that $\ld_A < \infty$. We then have
\begin{align*}
\phi(\chi_A)= \sum_{\af \in \Delta_A} \Theta_{e_{\af,\theta_\af(A)}, e_{\af, \theta_\af(A)}} \in \mathfrak{K}(X(\CB,\CL,\theta, \CI_\af)).
\end{align*}

For the "only if" direction, we first claim that for each $ \eta \in \mathfrak{K}(X(\CB,\CL,\theta, \CI_\af))$, the set $\{\af \in \CL: \| \left\langle e_{\af,\theta_\af(A)}, \eta (e_{\af,\theta_\af(A)})\right\rangle\| \geq 1\}$ is finite. By (\ref{compact}), it is enough to check that for $\eta=\theta_{x,y}$. For each $\af \in \CL$ and $x=(x_\af), ~y=(y_\af) \in X(\CB,\CL,\theta, \CI_\af)$, it is straightforward to check that $ \left\langle e_{\af,\theta_\af(A)}, \theta_{x,y} (e_{\af, \theta_\af(A)})\right\rangle = x_{\af}y_{\af}^*\chi_{\theta_\af(A)}$. Then our claim follows since $\|x_\af y_\af^*\| \geq 1$ for only a finitely many $\af$'s.

Now choose $A \in \CB $ so that $\phi(\chi_A) \in \mathfrak{K}(X(\CB,\CL,\theta, \CI_\af))$. Then the set $\Delta_A=\{\af \in \CL: \theta_\af(A) \neq \emptyset \}$ must be finite because $ \| \left\langle e_{\af,\theta_\af(A)}, \phi(\chi_A)(e_{\af,\theta_\af(A)})\right\rangle \|= \| \chi_{\theta_\af(A)}\|=1$. 

\vskip 0.5pc

(2): It is enough to prove that for $A \in \CB$, we have $\chi_A \in J_{X(\CB,\CL,\theta, \CI_\af)}$ $\iff$ $A \in \CB_{reg}$. We first cliam that 
\begin{align}\label{perp}
\ker(\phi)^\perp=\overline{\operatorname{\operatorname{span}}}\{\chi_A: A \in \CB ~\text{such that}~ 0<\ld_B  ~\text{for any}~ \emptyset \neq B \in \CI_A\}.
\end{align}
Since $\ker(\phi)= \cap_{\af \in \CL} \ker(\phi_\af)$, we see that 
\begin{align*} 
\ker(\phi) &= \overline{\operatorname{span}}\{\chi_B: B \in \CB ~\text{with}~\theta_\af(B)=\emptyset  ~\text{for all}~ \af \in \CL\} \\
&= \overline{\operatorname{span}}\{\chi_B: B \in \CB ~\text{with}~ \Delta_B=\emptyset \}.
\end{align*}
It then follows that  
\[
\ker(\phi)^\perp =\overline{\operatorname{span}}\{\chi_A: A \cap B=\emptyset ~\text{for all}~ B \in \CB ~\text{with}~ \Delta_B=\emptyset\}.
\]
We prove (\ref{perp}) by showing that 
\begin{multline*}  
\{A \in \CB :A \cap B=\emptyset ~\text{for all}~ B \in \CB ~\text{with}~ \Delta_B=\emptyset \} \\
=\{A \in \CB : 0<\ld_B  ~\text{for any}~ \emptyset \neq B \in \CI_A\}.
\end{multline*}

For the inclusion $(\subseteq)$, we choose $\emptyset \neq B \in \CI_A$. Then $\emptyset \neq B= A \cap B$, so $\Delta_B \neq \emptyset$, which means that $\ld_B > 0$. For the reverse inclusion $(\supseteq)$, choose $B \in \CB$ with $\Delta_B = \emptyset$.   Since $ \theta_\af(A \cap B) \subseteq \theta_\af(B) =\emptyset$ for all $\af \in \CL$, we have $\Delta_{A\cap B}=\emptyset$, meaning $\ld_{A \cap B}=0$.  Thus  $A \cap B = \emptyset$. 

Since for $A \in \CB$, it is obvious that $\ld_A < \infty \iff \ld_B < \infty ~\text{for all}~ \emptyset \neq B \in \CI_A $. Then using the result (1) and the equation (\ref{perp}), we have our result.
\end{proof}

Recall that a {\em (Toeplitz) representation} (\cite{FR} or \cite[Definition 2.1]{Ka2004c}) of the $C^*$-correspondence $X(\CB,\CL,\theta, \CI_\af)$ over $ \CA(\CB,\CL,\theta)$ on a $C^*$-algebra $\CX$ is a pair $(\pi, t)$ where $\pi: \CA(\CB,\CL,\theta) \to \CX$ is a $*$-homomorphism and $t: X(\CB,\CL,\theta, \CI_\af) \to \CX$ is a linear map satisfying
\begin{enumerate}
\item $t(x)^*t(y)=\pi(\left\langle x,y\right\rangle)$ for $x,y \in X(\CB,\CL,\theta, \CI_\af) $,
\item $\pi(f)t(x)=t(\phi(f)x)$ for $f \in \CA(\CB,\CL,\theta) $ and $x \in X(\CB,\CL,\theta, \CI_\af) $.
\end{enumerate}
Given a representation $(\pi,t)$, there is a $*$-homomorphism $\psi_t: \mathfrak{K}(X(\CB,\CL,\theta, \CI_\af)) \to \CX $ such that 
\[
\psi_t(\Theta_{x,y})=t(x)t(y)^*
\]
for $x,y \in X(\CB,\CL,\theta, \CI_\af)$. Let $K$ be an ideal in $J(X(\CB,\CL,\theta, \CI_\af))$. A representation $(\pi,t)$ is called {\em $K$-coisometric} if it satisfies that
\[
\pi(f)=\psi_t(\phi(f))
\]
for all $f \in K$ (\cite[Theorem 2.19]{MS1998} or \cite[Definition 1.1]{FMR2003}). The {\em relative Cuntz-Pimsner algebra $\CO_{(K, X(\CB,\CL,\theta, \CI_\af))}$ determined by the ideal $K$} is the $C^*$-algebra generated by a universal $K$-coisometric representation of $X(\CB,\CL,\theta, \CI_\af)$.

\begin{remark} 
When $K$ is the zero ideal of $J(X(\CB,\CL,\theta, \CI_\af))$, the algebra $\CO_{(K,X(\CB,\CL,\theta, \CI_\af))}$ is the Toeplitz algebra $\CT_{X(\CB,\CL,\theta, \CI_\af)}$. 
\end{remark}

\begin{remark}\label{ideal-bijective} 
For an ideal $\CJ$ of $\CB$, we let $K_\CJ$ denote the ideal of $\CA(\CB,\CL,\theta)$ generated by $\{\chi_A: A \in \CJ\}$, that is,   $$K_\CJ:=\overline{\operatorname{span}}\{\chi_A : A \in \CJ \}.$$ Then the map $\CJ \mapsto K_\CJ$ is a bijective map between the set of all ideals of $\CB$ and the set of all ideals of $\CA(\CB,\CL,\theta)$: For  an ideal $K$ of $\CA(\CB,\CL,\theta)$, we define $\CJ_K:=\{A \in \CB: \chi_A \in K\}$. Then it is easy to check that $\CJ_K$ is an ideal of $\CB$ and $K_{\CJ_K}=\overline{\operatorname{span}}\{\chi_A: A \in \CJ_K\}=K$ since $K=\overline{\operatorname{span}}\{\chi_A: A \in \CB, \chi_A \in K \}$ by Lemma \ref{prop: A(B,L)}(2). Thus the map is onto. The injectivity of the map is  rather obvious since we have $\CJ_{K_\CJ}=\CJ$. Note that $K \mapsto \CJ_K$ is the inverse map of $\CJ \mapsto K_\CJ$.
\end{remark}

The following theorem asserts that for a relative generalized Boolean dynamical system $(\CB,\CL,\theta,\CI_\af;\CJ)$, there exists a universal  $(\CB,\CL,\theta, \CI_\af;\CJ)$-representation. 

\begin{thm}\label{universal property:RGBDS} 
Let $(\CB,\CL,\theta,\CI_\af;\CJ)$ be a relative generalized Boolean dynamical system. Then there is a one-to-one correspondence between $K_{\CJ}$-coisometric representations of $X(\CB,\CL,\theta, \CI_\af)$ and $(\CB,\CL,\theta, \CI_\af ; \CJ)$-representations that takes a $K_\CJ$-coisometric representation $(\pi,t)$ of $X(\CB,\CL,\theta, \CI_\af)$ to the  $(\CB,\CL,\theta, \CI_\af ; \CJ)$-representation $\{\pi(\chi_A), t(e_{\af,B}): A \in  \CB, \af \in \CL  ~\text{and}~ B \in \CI_\af \}$.
 Thus, we have $ C^*(\CB,\CL,\theta,  \CI_\af;\CJ)  \cong \CO_{(K_\CJ, X(\CB,\CL,\theta, \CI_\af))}$.
\end{thm}

\begin{proof} 
Let $(\pi,t)$ be a $K_\CJ$-coisometric representation of $X(\CB,\CL, \theta, \CI_\af)$. We show that 
\[
\{\pi(\chi_A),  t(e_{\af,B}): A \in \CB, \af \in \CL ~\text{and}~ B \in \CI_\af \}
\]
is a  $(\CB,\CL,\theta, \CI_\af;\CJ)$-representation: It is straightforward to check that $\pi(\chi_A)$'s satisfy Definition \ref{def:representation of RGBDS}(i). For any $\af, \af' \in \CL$, $B \in \CI_\af$ and $B' \in \CI_{\af'}$, we observe for Definition \ref{def:representation of RGBDS}(ii) and (iii) that 
\[
\pi(\chi_A)t(e_{\af,B}) =t(\phi(\chi_A)e_{\af,B})
=t(e_{\af, B}\chi_{\theta_\af(A)})=t(e_{\af, B})\pi(\chi_{\theta_\af(A)})
\]
and that
\[
t(e_{\af,B})^*t(e_{\af',B'})
=\pi(\left\langle e_{\af,B},e_{\af',B'}\right\rangle)  
= \delta_{\af,\af'}\pi(\chi_{B \cap B'}). 
\]
For Definition \ref{def:representation of RGBDS}(iv), choose $A \in \CJ$. Then $\chi_A \in K_\CJ$. Hence, it follows that 
\begin{align*}
\pi(\chi_A)&=\psi_t(\phi(\chi_A)) \\
&=\psi_t(\sum_{\af \in \Delta_A}\Theta_{e_{\af,\theta_\af(A)}, e_{\af,\theta_\af(A)}})\\
&=\sum_{\af \in \Delta_A}t(e_{\af,\theta_\af(A)})t(e_{\af,\theta_\af(A)})^*. 
\end{align*}
Thus, $\{\pi(\chi_A), t(e_{\af,B})\}$ is a  $(\CB,\CL,\theta, \CI_\af;\CJ)$-representation.

For the opposite direction, suppose that $\{P_A, S_{\af,B}: A \in \CB, \af \in \CL ~\text{and}~ B \in \CI_\af \}$ is a  $(\CB,\CL,\theta, \CI_\af;\CJ)$-representation in a $C^*$-algebra $\CX$. Since we have  $\CA(\CB,\CL,\theta)=\overline{\operatorname{span}}\{\chi_A: A \in \CB\}$ and 
\begin{align}\label{X-span} 
X(\CB,\CL, \theta, \CI_\af)=\overline{\operatorname{span}}\{e_{\af,
A}: \af \in \CL ~\text{and}~A \in F ~\text{for some}~ F \in \CF\},
\end{align} 
where $\CF$ is the collection of finite subsets $F$ of mutually disjoint elements in $\CI_\af$, there can be at most one representation $(\pi,t)$ of $X(\CB,\CL,\theta, \CI_\af)$ such that $\pi(\chi_A)=P_A$  and $t(e_{\af,B})=S_{\af,B}$ for $A \in \CB$, $\af \in \CL$ and $B \in \CI_\af$. We shall construct such a representation $(\pi,t)$.

As \cite[Lemma 3.2]{BaCaPa2017}, we can see that there is a unique $*$-homomorphism 
\[
\pi: \CA(\CB,\CL,\theta) \to \mathcal{X}~\text{given by}~ \pi(\chi_A)=P_A
\] 
for $A \in \CB$. We construct a linear map $t: X(\CB,\CL,\theta, \CI_\af) \to  \CX $. Given a finite subset $F$ in $\CF$ of mutually disjoint elements in $\CI_\af$ and a finite subset $F'$ of $\CL$, we define 
\[
t \Big(\sum_{\af \in F', B \in \CF } c_{\af,B}e_{\af,B} \Big):=\sum_{\af \in F', B \in F } c_{\af,B}S_{\af,B},\]
where $c_{\af,B} \in \C$. Then we see that 
\begin{align*}  
t &\Big( \sum_{\af \in F', B \in F} c_{\af,B} e_{\af,B} \Big)^*t \Big( \sum_{\af' \in F', B' \in F}  d_{\af',B'} e_{\af', B'}\Big) \\
&= \Big(  \sum_{\af \in F', B \in F} c_{\af,B} S_{\af,B}\Big)^* \Big( \sum_{\af' \in F', B' \in F} d_{\af', B'} S_{\af',B'} \Big)\\
&=\sum_{\af \in F' ~\text{and}~ B,B' \in  F} \overline{c_{\af,B}}d_{\af, B'}  S_{\af,B}^*S_{\af,B'}\\
&= \sum_{\af \in F' ~\text{and}~ B \in  F} \overline{c_{\af,B}}d_{\af, B}  P_B\\
&=\sum_{\af \in F' ~\text{and}~ B,B' \in F} \overline{c_{\af,B}}d_{\af, B'}  \pi(\chi_{B\cap B'}) \\
&=\sum_{\af \in F' ~\text{and}~ B,B' \in F} \overline{c_{\af,B}}d_{\af, B'}  \pi(\chi_B)\pi(\chi_{B'})  \\
&=\pi \Big( \bigl\langle   \sum_{\af \in F', B \in F} c_{\af,B} e_{\af,B} ,  \sum_{\af' \in F', B' \in F}  d_{\af',B'} e_{\af', B'} \bigr\rangle\Big).
\end{align*}
Thus, $t$ extends to a linear map from $X(\CB,\CL,\theta, \CI_\af)$ to $ \CX$ which satisfies 
\[
t(x)^*t(y)=\pi(\left\langle x,y\right\rangle)
\] 
for all $x, y \in X(\CB,\CL,\theta, \CI_\af)$. We next show that we have 
\[
\pi(f)t(x)=t(\phi(f)x)
\]  
for $f \in \CA(\CB,\CL, \theta)$ and $x \in X(\CB,\CL,\theta, \CI_\af)$. For $A \in \CB$,  $\af \in \CL$ and $B \in \CI_\af$, we have 
\[\pi(\chi_A)S_{\af, B}=P_A S_{\af,B}= S_{\af,B}P_{\theta_\af(A)}=S_{\af,B}\pi(\phi_\af(\chi_A)).
\]
Since $\CA(\CB,\CL,\theta)=\overline{\operatorname{span}}\{\chi_A:A \in \CB \}$, we have $\pi(f)S_{\af,B}=S_{\af,B}\pi(\phi_\af(f))$ for any $f \in \CA(\CB,\CL,\theta)$ and $\af \in \CL$. Thus we see that 
\begin{align*}
\pi(f)t(e_{\af,B}) =\pi(f)S_{\af,B}&=S_{\af,B}\pi(\phi_\af(f)) \\&=t(e_{\af,B})\pi(\phi_\af(f))=t(\phi(f)e_{\af,B}) 
\end{align*}
for all $f \in \CA(\CB,\CL,\theta)$,  $\af \in \CL$ and $B \in \CI_\af$. Our claim then follows from (\ref{X-span}). 

If $A\in\CJ$, then
\begin{align*} 
\pi(\chi_A)=P_A &= \sum_{\af \in \Delta_A} S_{\af, \theta_\af(A)} S_{\af, \theta_\af(A)}^* \\
&= \sum_{\af \in \Delta_A} t(e_{\af,\theta_\af(A)}) t(e_{\af,\theta_\af(A)})^*\\
&=  \sum_{\af \in \Delta_A} \psi_t(\Theta_{e_{\af,\theta_\af(A)},e_{\af,\theta_\af(A)}}) \\
&=\psi_t \Big(\sum_{\af \in \Delta_A} \Theta_{e_{\af,\theta_\af(A)},e_{\af,\theta_\af(A)}} \Big) \\
&=\psi_t(\phi(\chi_A)). 
\end{align*}
Hence, $\pi(f)=\psi_t(\phi(f))$ for all $f \in K_\CJ$. We have proved that $(\pi,t)$ is a $K_\CJ$-coisometric representation of $X(\CB,\CL,\theta, \CI_\af)$. 

Therefore, we have $ C^*(\CB,\CL,\theta, \CI_\af; \CJ) \cong   \CO_{(K_\CJ, X(\CB,\CL,\theta, \CI_\af))} $ by the universal nature of each algebra. 
\end{proof}

\begin{cor}\label{existence of $C^*$-BDS} 
Let $(\CB,\CL,\theta,\CI_\alpha)$ be a  generalized Boolean dynamical system. Then we have
\begin{enumerate}
\item $ \CT(\CB,\CL,\theta, \CI_\af)  \cong \CT_{X(\CB,\CL,\theta, \CI_\af)}$,
\item $C^*(\CB,\CL,\theta, \CI_\af)  \cong \CO_{X(\CB,\CL,\theta, \CI_\af)}$.
\end{enumerate}
\end{cor}

\begin{proof}
(1): Taking $\CJ=\emptyset$ in Theorem \ref{universal property:RGBDS}, we have 
\[C^*(\CB,\CL,\theta,  \CI_\af;\emptyset) \cong \CO_{(K_\emptyset, X(\CB,\CL,\theta,  \CI_\af))}=\CO_{(\{0\}, X(\CB,\CL,\theta, \CI_\af))}=\CT_{X(\CB,\CL,\theta, \CI_\af)}.\]

\noindent
(2): If $\CJ=\CB_{reg}$, then $K_{\CB_{reg}}=\overline{\operatorname{span}}\{\chi_A: A \in \CB_{reg}\}=J_{X(\CB,\CL,\theta, \CI_\af)}$ by Lemma \ref{Jx}(2). Thus it  follows that 
\[
C^*(\CB,\CL,\theta, \CI_\af; \CB_{reg}) \cong \CO_{(K_{\CB_{reg}}, X(\CB,\CL,\theta, \CI_\af))}= \CO_{ X(\CB,\CL,\theta, \CI_\af)}.
\]
\end{proof}

\begin{cor}\label{cor:inj}
Let $(\CB,\CL,\theta,\CI_\af;\CJ)$ be a relative generalized Boolean dynamical system. Let $\{p_A,s_{\alpha,B}:A\in\CB,\ \alpha\in\CL,\ B\in\CI_\alpha\}$ be the universal $(\CB,\CL,\theta,\CI_\alpha;\CJ)$-representation. Then $p_A\ne 0$ for $A\in\CB \setminus \{\emptyset\}$, and $p_A=\sum_{\alpha\in\Delta_A}s_{\alpha,\theta_\alpha(A)}s_{\alpha,\theta_\alpha(A)}^*$ if and only if $A\in\CJ$.
\end{cor}

\begin{proof}
This follows from Theorem~\ref{universal property:RGBDS} and \cite[Corollary 11.8]{Ka2007}.
\end{proof}

\section{The gauge-invariant uniqueness theorem}\label{GIUT} 

In this section we shall prove a gauge-invariant uniqueness theorem of
 $C^*$-algebras of  relative generalized Boolean dynamical systems. We
also show that the class of $C^*$-algebras of relative generalized
Boolean dynamical systems is the same as the class of $C^*$-algebras
of generalized Boolean dynamical systems, and that the $C^*$-algebra
of any relative generalized Boolean dynamical system is Morita
equivalent to the $C^*$-algebra of a Boolean dynamical system.

\begin{thm}[Gauge-Invariant Uniqueness for $C^*(\CB,\CL, \theta, \CI_\af;\CJ)$] 
\label{GIUT for RGBDS}  

Let $(\CB,\CL,\theta,\CI_\af;\CJ)$ be a relative generalized Boolean dynamical system and let $C^*(\CB,\CL, \theta, \CI_\af;\CJ)=C^*(p_A, s_{\af,B})$.
Suppose that $\{P_A, S_{\af,B}: A \in \CB,~ \af \in \CL ~\text{and}~ B \in \CI_\af \}$ is a   $(\CB, \CL, \theta, \CI_\af;\CJ)$-representation in a $C^*$-algebra $\CX$ and let $\pi: C^*(\CB,\CL,\theta, \CI_\af;\CJ) \to \CX$ be the unique $*$-homomorphism such that  $\pi(p_A)=P_A$ and $\pi(s_{\af,B})=S_{\af,B}$ for all $ A \in \CB$, $\af \in \CL $ and $B \in \CI_\af$. 
Then  $\pi$ is injective if and only if the following properties hold:
\begin{enumerate} 
\item $P_A \neq 0$ whenever $A \neq \emptyset$, 
\item $P_A  -\sum_{\af \in \Delta_A} S_{\af, \theta_\af(A)}S_{\af, \theta_\af(A)}^*  \neq 0$ for all $A \in \CB_{reg} \setminus \CJ$,
\item there exists for each $z\in\T$ a $*$-homomorphism $\bt_z: C^*(P_A,
S_{\af,B}) \to  C^*(P_A, S_{\af,B}) $ such that $\bt_z(P_A)=P_A$ and
$\bt_z(S_{\af,B})=zS_{\af,B}$ for $A \in \CB$, $\af \in \CL$ and $B
\in \CI_\af$.
\end{enumerate}
\end{thm}

\begin{proof}
If $\{P_A, S_{\af,B}: A \in \CB,~ \af \in \CL ~\text{and}~ B \in
\CI_\af \}$ is a $(\CB, \CL, \theta, \CI_\af;\CJ)$-representation in a
$C^*$-algebra $\CX$, then by Theorem \ref{universal property:RGBDS},
we have a $K_\CJ$-coisometric representation $(\pi',t')$ of
$X(\CB,\CL,\theta, \CI_\af)$. The conditions (1)-(3) are equivalent to
(1)'-(3)' stated below, respectively:
 \begin{enumerate}
\item[(1)']  $\ker \pi'=0$,
\item[(2)'] $\{a \in \CA(\CB,\CL,\theta): \phi(a)\in
\CK(X(\CB,\CL,\theta,\CI_\af))\text{ and } \pi'(a)=
\psi_{t'}(\phi(a))\} = K_\CJ$,
\item[(3)'] $(\pi',t')$ admits a gauge action.
\end{enumerate}
Thus, the result follows from \cite[Corollary 11.8]{Ka2007}.
\end{proof}

As a corollary, we obtain a version of the  gauge-invariant uniqueness theorem for $C^*$-algebras of generalized Boolean dynamical systems. 
It is a generalization of \cite[Corollary 3.10]{BaCaPa2017}, \cite[Theorem 5.10]{COP}, and \cite[Theorem 4.2]{JKP}.

\begin{cor}[Gauge-Invariant Uniqueness for  $C^*(\CB,\CL,\theta,\CI_\alpha)$] \label{GIUT of GBDS} \label{GIUT for GBDS}
Let $(\CB,\CL,\theta,\CI_\af)$ be a generalized Boolean dynamical system, and let $C^*(\CB,\CL, \theta, \CI_\af):=C^*(p_A, s_{\af,B})$. 
Suppose that $\{P_A, S_{\af,B}: A \in \CB,~ \af \in \CL ~\text{and}~ B \in \CI_\af \}$ is a Cuntz--Krieger representation of $(\CB, \CL, \theta, \CI_\af)$ in a $C^*$-algebra $\CX$ and let $\pi: C^*(\CB,\CL,\theta, \CI_\af) \to \CX$ be the unique $*$-homomorphism such that  $\pi(p_A)=P_A$ and $\pi(s_{\af,B})=S_{\af,B}$ for all $ A \in \CB$, $\af \in \CL $ and $B \in \CI_\af$. Then $\pi$ is injective if and only if   the following properties hold:
\begin{enumerate} 
\item $P_A \neq 0$ whenever $A \neq \emptyset$, 
\item there exists for each $z\in\T$ a $*$-homomorphism $\bt_z: C^*(P_A,
S_{\af,B}) \to  C^*(P_A, S_{\af,B}) $ such that $\bt_z(P_A)=P_A$ and
$\bt_z(S_{\af,B})=zS_{\af,B}$ for $A \in \CB$, $\af \in \CL$ and $B
\in \CI_\af$.
\end{enumerate}
\end{cor}

\begin{proof} Taking $\CJ=\CB_{reg}$ in Theorem \ref{GIUT for RGBDS}, we have the result. 
\end{proof}

The following proposition has the consequence that the $C^*$-algebra
of any relative generalized Boolean dynamical system $(\CB,\CL,\theta,
\CI_\af;\CJ)$ is Morita equivalent to $C^*(\CB,\CL,\theta;\CJ)
(=C^*(\CB,\CL,\theta, \CR_\af;\CJ))$.

\begin{prop}\label{full hereditary} 
Let $(\CB,\CL,\theta,\CI_\af;\CJ)$ be a relative generalized Boolean dynamical system. Then $C^*(\mathcal{B},\mathcal{L},\theta,\mathcal{I}_\alpha;\CJ)$  is
(isomorphic to) a full hereditary $C^*$-subalgebra of $C^*(\mathcal{B},\mathcal{L},\theta,\mathcal{B};\CJ)$.   
\end{prop}

\begin{proof} 
Let $C^*(\mathcal{B},\mathcal{L},\theta,\mathcal{I}_\alpha;\CJ):=C^*(q_A,t_{\af,B})$ and $C^*(\mathcal{B},\mathcal{L},\theta,\mathcal{B};\CJ):=C^*(p_A,s_{\af,B})$. Since $\{p_A,s_{\af,B}: A\in \CB, \af \in \CL ~\text{and}~ B \in \CI_\af\}$ is a    $(\mathcal{B},\mathcal{L},\theta,\mathcal{I}_\alpha;\CJ)$-representation,  by the universal property of  $C^*(\mathcal{B},\mathcal{L},\theta,\mathcal{I}_\alpha;\CJ)$ there exists a $*$-homomorphism $\iota:C^*(\mathcal{B},\mathcal{L},\theta,\mathcal{I}_\alpha;\CJ) \to C^*(\mathcal{B},\mathcal{L},\theta,\mathcal{B};\CJ) $ such that  
\[
\iota(q_A) = p_A ~\text{and }~ \iota(t_{\af,B})=s_{\af,B}
\] 
for $A \in \CB$, $\af \in \CL$ and $B\in \CI_\af$. Since $ C^*(\mathcal{B},\mathcal{L},\theta,\mathcal{B};\CJ) $ admits a gauge action, $\iota(q_A)=p_A \neq 0$ for $A \neq \emptyset$, and $p_A  -\sum_{\af \in \Delta_A} s_{\af, \theta_\af(A)}s_{\af, \theta_\af(A)}^*  \neq 0$ for all $A \in \CB_{reg} \setminus \CJ$ by Corollary \ref{cor:inj}, 
 the homomorphism $\iota$ is injective by Theorem  \ref{GIUT for RGBDS}. 

We identify $\iota(q_A)$ with $p_A$  and $\iota(t_{\af,B})$ with $ s_{\af,B}$ so that  $C^*(\mathcal{B},\mathcal{L},\theta,\mathcal{I}_\alpha;\CJ)$ is a subalgebra of $C^*(\mathcal{B},\mathcal{L},\theta,\mathcal{B};\CJ)$. We then have that $C^*(\mathcal{B},\mathcal{L},\theta,\mathcal{I}_\alpha;\CJ)$ is hereditary because $p_As_{\alpha,B}p_B=s_{\alpha,B\cap\theta_\alpha(A)}\in C^*(\mathcal{B},\mathcal{L},\theta,\mathcal{I}_\alpha;\CJ)$ for any $A,B\in\mathcal{B}$ and any $\alpha\in\mathcal{L}$, and it is full because any ideal of $C^*(\mathcal{B},\mathcal{L},\theta,\mathcal{B})$ that contains $\{p_A:A\in\mathcal{B}\}$ must be equal to $C^*(\mathcal{B},\mathcal{L},\theta,\mathcal{B};\CJ)$ since $s_{\af,B}p_B=s_{\af,B}$ for any $B \in \CB$ and $\af \in \CL$. 
\end{proof}

Let $(\CB,\CL,\theta,\CI_\af;\CJ)$ be a relative generalized Boolean
dynamical system. By imitating a construction in \cite[Section
6]{Ka2007}, we shall now construct a generalized Boolean dynamical system
$(\widetilde{\CB},\CL,\widetilde{\theta},\widetilde{\CI}_\alpha)$ such
that $C^*(\CB,\CL,\theta,\CI_\af;\CJ)$ is isomorphic to
$C^*(\widetilde{\CB},\CL,\widetilde{\theta},\widetilde{\CI}_\alpha)$.

Given a relative generalized Boolean dynamical system  $(\CB,\CL,\theta,\CI_\af;\CJ)$, we let 
\[
\widetilde{\CB}=\{(A,[B]_{\CJ}):A,B\in\CB ~\text{and}~ [A]_{\CB_{reg}}=[B]_{\CB_{reg}}\}.
\] 
Then $\widetilde{\CB}$ is a Boolean algebra with empty set $(\emptyset, \emptyset):=(\emptyset, [\emptyset]_\CJ)$ and operations defined by
\begin{align*} 
(A_1,[B_1]_{\CJ})\cup (A_2,[B_2]_{\CJ})&:=(A_1\cup A_2,[B_1\cup B_2]_{\CJ}),\\
(A_1,[B_1]_{\CJ})\cap (A_2,[B_2]_{\CJ})&:=(A_1\cap A_2,[B_1\cap B_2]_{\CJ}),\\
(A_1,[B_1]_{\CJ})\setminus (A_2,[B_2]_{\CJ})&:=(A_1\setminus A_2,[B_1\setminus B_2]_{\CJ}).
\end{align*}  
For $\alpha\in\CL$, we define $\widetilde{\theta}_\alpha:\widetilde{\CB} \to \widetilde{\CB}$ by 
\[
\widetilde{\theta}_\alpha(A,[B]_{\CJ})=(\theta_\alpha(A),[\theta_\alpha(A)]_{\CJ}).
\]
Then $(\widetilde{\CB}, \CL,\widetilde{\theta})$ is a Boolean
dynamical system. Since $\Delta^{(\widetilde{\CB},
\CL,\widetilde{\theta})}_{(\emptyset,[B]_\CJ)}=\emptyset$ for any
$B\in\CB$, we have that $[B]_\CJ=\emptyset$ if $(A,[B]_\CJ)\in
\widetilde{\CB}^{(\widetilde{\CB}, \CL,\widetilde{\theta})}_{reg}$. It follows
that
\begin{equation*}
\widetilde{\CB}_{reg}:=\widetilde{\CB}^{(\widetilde{\CB},
\CL,\widetilde{\theta})}_{reg}=\{(A,\emptyset):A\in\CB_{reg}\}.
\end{equation*}

\begin{prop}\label{RGBDS isom GBDS}
Let $(\CB,\CL,\theta,\CI_\af;\CJ)$ be a relative generalized Boolean dynamical system. Then we have 
\[
C^*(\CB,\CL,\theta,\CI_\af;\CJ) \cong C^*(\widetilde{\CB},\CL,\widetilde{\theta},\widetilde{\CI}_\alpha),
\]
where $\widetilde{\CI}_\alpha=\{(A,[A]_{\CJ}):A\in\CI_\alpha\}$ for $\alpha\in\CL$.
\end{prop}

\begin{proof} 
Let $\{p_A, s_{\af,B}: A \in \CB, \af \in \CL ~\text{and}~ B \in \CI_\af\}$ be the universal  $(\CB,\CL,\theta,\CI_\af;\CJ)$-representation, and let $\{p_{(A,[B]_\CJ)}, s_{\af,(A,[A]_\CJ)}: (A,[B]_\CJ) \in \widetilde{\CB}, \af \in \CL ~\text{and}~ (A,[A]_\CJ) \in \widetilde{\CI}_\af\}$ be the universal Cuntz-Krieger representation of $(\widetilde{\CB},\CL,\widetilde{\theta},\widetilde{\CI}_\alpha)$. It is easy to see that  $\{p_{(A,[A]_\CJ)}, s_{\af,(B,[B]_\CJ)}:A \in \CB, \af \in \CL ~\text{and}~ B \in \CI_\af\}$ is a  $(\CB, \CL, \theta, \CI_\af;\CJ)$-representation. There is therefore a $*$-homomorphsim 
\[
\phi:C^*(\CB,\CL,\theta,\CI_\af;\CJ) \to C^*(\widetilde{\CB},\CL,\widetilde{\theta},\widetilde{\CI}_\alpha)
\]
such that
\[
\phi(p_A)=p_{(A,[A]_\CJ)}~\text{and}~ \phi(s_{\af,B})=s_{\af,(B,[B]_\CJ)}
\]
for $A \in \CB$, $\af \in \CL$ and $ B \in \CI_\af$.

We shall construct an inverse to $\phi$. For each $(A,[B]_\CJ) \in \widetilde{\CB}$, $\af \in \CL$ and $ (A,[A]_\CJ) \in \widetilde{\CI}_\af$, we define
\begin{align*} 
P_{(A,[B]_\CJ)} &:=p_A+p_C-\sum_{\alpha\in\Delta_C}s_{\alpha,\theta_\alpha(C)}s_{\alpha,\theta_\alpha(C)}^*-p_D+\sum_{\alpha\in\Delta_D}s_{\alpha,\theta_\alpha(D)}s_{\alpha,\theta_\alpha(D)}^*, \\
S_{\af,(A,[A]_\CJ)}&:=s_{\af,A}, 
\end{align*}
where $C,D\in\CB_{reg}$ are such that $A\cup C=B\cup D$ and $A \cap C =B \cap D =\emptyset$. 

We first prove that the projection $P_{(A,[B]_\CJ)}$ does not depend on the choice of $B$, $C$, and $D$. Supposet that $[B]_\CJ=[B']_\CJ$, $A\cup C=B\cup D$ with  $A\cap C=B\cap D =\emptyset$, and $A\cup C'=B'\cup D'$ with $A\cap C'=B'\cap D' =\emptyset$. To ease notion we let 
\[
q_A:=p_A-\sum_{\alpha\in\Delta_A}s_{\alpha,\theta_\alpha(A)}s_{\alpha,\theta_\alpha(A)}^*
\] 
when $A\in\CB_{reg}$. Then $q_A=0$ if $A\in\CJ$. In particular, if $A\subseteq B\setminus B'$ or $A\subseteq B'\setminus B$, then $q_A=0$. Thus,  
\[
q_{C'\setminus (C\cup D')}=q_{C\cap C'\cap D'\setminus D}=q_{D'\setminus (C'\cup D)}=0
\] 
since $C'\setminus (C\cup D') \subseteq B' \setminus B$, $C\cap C'\cap D'\setminus D \subseteq B \setminus B'$ and $D'\setminus (C'\cup D) \subseteq B' \setminus B$. We also have that $q_{D\cap C'\cap D'\setminus C}=0$ because $D\cap C'\cap D'\setminus C\subseteq A\cap C'=\emptyset$. It  then follows that 
\begin{align*}
q_{C'\setminus C}&=q_{C'\setminus (C\cup D')}+q_{C'\cap D'\setminus C}=q_{C'\cap D'\setminus C}=q_{D\cap C'\cap D'\setminus C}+q_{C'\cap D'\setminus (C\cup D)} \\
&=q_{C'\cap D'\setminus (C\cup D)}=q_{C'\cap D'\setminus (C\cup D)}+q_{C\cap C'\cap D'\setminus D}=q_{C'\cap D'\setminus D} \\
&=q_{C'\cap D'\setminus D}+q_{D'\setminus (C'\cup D)}=q_{D'\setminus D}. 
\end{align*}
Thus, $q_C+q_{D\cup D'}=q_C+q_D+q_{D'\setminus D}=q_C+q_D+q_{C'\setminus C}=q_{C\cup C'}+q_D$ from which we have that 
\[
p_A+q_C-q_D=p_A+q_{C\cup C'}-q_{D\cup D'}=p_A+q_{C'}-q_{D'}.
\]
So, $P_{(A,[B]_\CJ)}=P_{(A,[B']_\CJ)}$. 

\vskip 0.5pc

The following calculations show that the family 
$\{P_{(A,[B]_\CJ)}, S_{\af,(A,[A]_\CJ)}\}$ is a Cuntz--Krieger representation of $(\widetilde{\CB},\CL,\widetilde{\theta},\widetilde{\CI}_\alpha)$:

(i): Clearly, $P_{(\emptyset, \emptyset)}=0$. Given $(A_1,[B_1]_\CJ)$ and $(A_2, [B_2]_\CJ)$ in $\widetilde{\CB}$, we can choose $C_1,D_1,C_2,D_2 \in \CB_{reg}$ such that 
\begin{align*} & A_1\cup C_1=B_1\cup D_1 ~\text{with}~A_1 \cap  C_1=B_1\cap D_1=\emptyset, \\
& A_2\cup C_2=B_2\cup D_2 ~\text{ with}~ A_2\cap C_2=B_2\cap D_2 =\emptyset,
\end{align*}
and
\begin{align*} 
&(A_1 \cap A_2) \cup \{ (A_1\cap C_2) \cup (D_1 \cap D_2) \cup (C_1 \cap A_2)\cup (C_1 \cap C_2)\} \\
& =(B_1 \cap B_2) \cup \{ (C_1 \cap D_2) \cup (D_1 \cap A_2) \cup (D_1\cap C_2) \cup (A_1\cap D_2)\},
\end{align*}
with
\begin{align*} 
&(A_1 \cap A_2) \cap \{ (A_1\cap C_2) \cup (D_1 \cap D_2) \cup (C_1 \cap A_2)\cup (C_1 \cap C_2)\} = \emptyset ,\\
& (B_1 \cap B_2) \cap \{ (C_1 \cap D_2) \cup (D_1 \cap A_2) \cup (D_1\cap C_2) \cup (A_1\cap D_2)\} =\emptyset.
\end{align*}
Using the fact that $p_Aq_B=q_{A \cap B}$ for all $A \in \CB$ and $B\in \CB_{reg}$, it easily follows that  
\begin{align*} 
&P_{(A_1,[B_1]_\CJ)}P_{(A_2, [B_2]_\CJ)}\\
&=\big(p_{A_1}+q_{C_1}-q_{D_1}\big) \big(p_{A_2}+q_{C_2}-q_{D_2}\big)\\
&= p_{A_1 \cap A_2} 
+q_{A_1\cap C_2} +q_{D_1 \cap D_2}+q_{C_1 \cap A_2}+q_{C_1 \cap C_2} \\
&\hskip 4pc-q_{C_1 \cap D_2}-q_{D_1 \cap A_2}-q_{D_1\cap C_2} -q_{A_1\cap D_2} \\
&=P_{(A_1 \cap A_2, [B_1 \cap B_2]_\CJ)}.
\end{align*}
The last equality holds true since one can choose a finite family $\{B_i\}_{i=1}^n$ of mutually disjoint elements in $\CB_{reg}$ so that 
\[
(A_1\cap C_2) \cup (D_1 \cap D_2) \cup (C_1 \cap A_2)\cup (C_1 \cap C_2) = \cup_{i=1}^n B_i.
\]

Let $X_1=A_1\cup C_1=B_1\cup D_1$ and $X_2=A_2\cup C_2=B_2\cup D_2$. We then have that $(C_1\cap C_2)\cup (C_1\setminus X_2)\cup (C_2\setminus X_1), (D_1\cap D_2)\cup (D_1\setminus X_2)\cup (D_2\setminus X_1)\in\CB_{reg}$, 
\begin{align*}
&(A_1\cup A_2)\cap ((C_1\cap C_2)\cup (C_1\setminus X_2)\cup (C_2\setminus X_1))\\&=(B_1\cup B_2)\cap ((D_1\cap D_2)\cup (D_1\setminus X_2)\cup (D_2\setminus X_1))=\emptyset
\end{align*} 
and that 
\begin{align*}
&X_1\cup X_2\\ &=A_1\cup A_2\cup\big( (C_1\cap C_2)\cup (C_1\setminus X_2)\cup (C_2\setminus X_1) \big)\\&=B_1\cup B_2\cup \big( (D_1\cap D_2)\cup (D_1\setminus X_2)\cup (D_2\setminus X_1) \big).
\end{align*} It then follows that
\begin{align*}
&P_{(A_1\cup A_2,[B_1\cup B_2]_\CJ)}\\
&=p_{A_1\cup A_2}+q_{(C_1\cap C_2)\cup (C_1\setminus X_2)\cup (C_2\setminus X_1)}-q_{(D_1\cap D_2)\cup (D_1\setminus X_2)\cup (D_2\setminus X_1)}.
\end{align*}
Since $C_1 \setminus (A_2 \cup C_2)= C_1 \setminus \big((C_1 \cap A_2) \cup (C_1 \cap C_2)\big)$, it follows that
\begin{align*}
q_{C_1\setminus X_2}&=q_{C_1}-q_{C_1\cap A_2}-q_{C_1\cap C_2}.
\end{align*}
Similarly, we have 
\begin{align*}
q_{C_2\setminus X_1}&=q_{C_2}-q_{C_2\cap A_1}-q_{C_2\cap C_1},\\
q_{D_1\setminus X_2}&=q_{D_1}-q_{D_1\cap A_2}-q_{D_1\cap C_2},\\
q_{D_2\setminus X_1}&=q_{D_2}-q_{D_2\cap A_1}-q_{D_2\cap C_1}.
\end{align*}
We thus have that
\begin{align*}
&P_{(A_1,[B_1]_\CJ)}+P_{(A_2,[B_2]_\CJ)}-P_{(A_1\cap A_2,[B_1\cap B_2]_\CJ)}\\
&=(p_{A_1}+q_{C_1}-q_{D_1})+(p_{A_2}+q_{C_2}-q_{D_2}) \\
&  \hskip 1pc  -p_{A_1\cap A_2}-q_{C_1\cap A_2}-q_{C_2\cap A_1}-q_{C_1\cap C_2}-q_{D_1\cap D_2}\\
&  \hskip 5pc+q_{D_1\cap A_2}+q_{D_1\cap C_2}+q_{D_2\cap A_1}+q_{D_2\cap C_1}\\
&=(p_{A_1}+p_{A_2}-p_{A_1\cap A_2})\\
&  \hskip 1pc +q_{C_1\cap C_2}
+(q_{C_1}-q_{C_1\cap A_2}-q_{C_1\cap C_2})
+(q_{C_2}-q_{C_2\cap A_1}-q_{C_2\cap C_1}) \\
&  \hskip 1pc 
-q_{D_1\cap D_2}
-(q_{D_1}-q_{D_1\cap A_2}-q_{D_1\cap C_2})
-(q_{D_2}-q_{D_2\cap A_1}-q_{D_2\cap C_1})\\
&=p_{A_1\cup A_2}+q_{(C_1\cap C_2)\cup (C_1\setminus X_2)\cup (C_2\setminus X_1)}
-q_{(D_1\cap D_2)\cup (D_1\setminus X_2)\cup (D_2\setminus X_1)}\\
&=P_{(A_1\cup A_2,[B_1\cup B_2]_\CJ)}.
\end{align*}

(ii): Using $q_As_{\af,A'}=s_{\af, A' \cap \theta_\af(A)}-s_{\af, A' \cap \theta_\af(A)}=0$ for all $A\in \CB_{reg}$ and $A' \in \CI_\af$, we see that  
\begin{align*}
P_{(A, [B]_\CJ)}S_{\af, (A',[A']_\CJ)} &= (p_A+q_C-q_D)s_{\af,A'} \\
&=s_{\af,A'}p_{\theta_\af(A)}  \\
&=S_{\af, (A',[A']_\CJ)}P_{\theta_\af(A, [B]_\CJ)}.
\end{align*}

(iii): $S_{\af,(A, [A]_\CJ)}^*S_{\af',(A', [A']_\CJ)}^*=s_{\af, A}^*s_{\af',A'}=\delta_{\af,\af'}p_{A \cap A'}=\delta_{\af,\af'}P_{(A \cap A',[A \cap A']_\CJ)}$.

\vskip 0.5pc

(iv): For $(A, \emptyset) \in \widetilde{\CB}_{reg}$, we have
\begin{align*}
P_{(A, \emptyset)}=p_A &=\sum_{\af \in \Delta_{(A, \emptyset)}}s_{\af,\theta_\af(A)}s_{\af,\theta_\af(A)}^* \\&=\sum_{\af \in \Delta_{(A, \emptyset)}}S_{\af,(\theta_\af(A), [\theta_\af(A)]_\CJ)}S_{\af,(\theta_\af(A), [\theta_\af(A)]_\CJ)}^*.
\end{align*}

\noindent
Thus we have a $*$-homomorphsim 
\[
\rho : C^*(\widetilde{\CB},\CL,\widetilde{\theta},\widetilde{\CI}_\alpha) \to C^*(\CB,\CL,\theta,\CI_\af;\CJ)
\]
such that
\[
\rho(p_{(A,[B]_\CJ)})=p_A+q_C-q_D ~\text{and}~ \rho(s_{\af,(A,[A]_\CJ)})=s_{\af,A}
\]
for $(A,[B]_\CJ) \in \widetilde{\CB}$, $\af \in \CL$ and $ (A,[A]_\CJ) \in \widetilde{\CI}_\af$.

It is then easy to check that
\begin{align*} 
\rho \circ \phi (p_A) &= \rho (p_{(A,[A]_\CJ)}) =p_A, \\
\rho \circ \phi (s_{\af,B}) &=  \rho(s_{\af,(B,[B]_\CJ)})=s_{\af,B}
\end{align*}
for all  $A \in \CB$, $\af \in \CL$ and $ B \in \CI_\af$. Hence, $\rho \circ \phi =id$. Also, we see  that 
\begin{align*}
& \phi \circ \rho (p_{(A,[B]_\CJ)}) \\
&= \phi(p_A+q_C-q_D) \\
&=p_{(A,[A]_\CJ)}+p_{(C,[C]_\CJ)} -\sum_{\af \in \Delta_C}s_{\af, (\theta_\af(C), [\theta_\af(C)]_\CJ)}s_{\af, (\theta_\af(C), [\theta_\af(C)]_\CJ)}^* \\
& \hskip 4pc -p_{(D,[D]_\CJ)} +\sum_{\af \in \Delta_D}s_{\af, (\theta_\af(D), [\theta_\af(D)]_\CJ)}s_{\af, (\theta_\af(D), [\theta_\af(D)]_\CJ)}^* \\
&=p_{(A,[A]_\CJ)} \\
&\hskip 1pc+p_{(\emptyset, [C]_\CJ)}+p_{(C,\emptyset)} -\sum_{\af \in \Delta_{(C,\emptyset)}}s_{\af, (\theta_\af(C), [\theta_\af(C)]_\CJ)}s_{\af, (\theta_\af(C), [\theta_\af(C)]_\CJ)}^* \\
& \hskip 1pc -p_{(\emptyset,[D]_\CJ)}-p_{(D,\emptyset)} +\sum_{\af \in \Delta_{(D,\emptyset)}}s_{\af, (\theta_\af(D), [\theta_\af(D)]_\CJ)}s_{\af, (\theta_\af(D), [\theta_\af(D)]_\CJ)}^* \\
&=p_{(A,[A]_\CJ)}+p_{(\emptyset, [C]_\CJ)} -p_{(\emptyset,[D]_\CJ)}\\
&=p_{(A,[A \cup C]_\CJ)}-p_{(\emptyset,[D]_\CJ)}\\
&=p_{(A,[B]_\CJ)}.
\end{align*}
for  $(A,[B]_\CJ) \in \widetilde{\CB}$. The last equality holds since 
\[p_{(A, [A \cup C]_\CJ)}=p_{(A, [B \cup D]_\CJ)}=p_{(A, [B]_\CJ)}+ p_{(\emptyset, [D]_\CJ)}.
\]
For  $\af \in \CL$ and $ (A,[A]_\CJ) \in \widetilde{\CI}_\af$, we have  
\[
\phi \circ \rho (s_{\af,(A,[A]_\CJ)}) = \phi(s_{\af,A}) = s_{\af,(A,[A]_\CJ)}.
\]
Hence, $\phi \circ \rho =\operatorname{id}$.
\end{proof}

\begin{remark} Let $(\CB,\CL,\theta,\CI_\af;\CJ)$ be a relative generalized
Boolean dynamical system. By Proposition \ref{RGBDS isom GBDS},
$C^*(\CB,\CL,\theta,\CI_\alpha;\CJ)$ is isomorphic to
$C^*(\widetilde{\CB},\CL,\widetilde{\theta},\widetilde{\CI}_\alpha)$,
and it follows from Proposition \ref{full hereditary} that
$C^*(\widetilde{\CB},\CL,\widetilde{\theta},\widetilde{\CI}_\alpha)$
is Morita equivalent to $C^*(\widetilde{\CB},\CL,\widetilde{\theta})$.
We thus have that the $C^*$-algebra of any relative generalized
Boolean dynamical system is Morita equivalent to the $C^*$-algebra of
a Boolean dynamical system.

\end{remark}

\section{Gauge-invariant ideals in $C^*(\CB,\CL,\theta, \CI_\af; \CJ)$}
\label{gauge-invariant ideals}
In this section, we give a complete list of the gauge-invariant ideals of $C^*$-algebras of $(\CB,\CL, \theta, \CI_\af;\CJ)$ and describe the quotients as $C^*$-algebras of relative generalized Boolean dynamical systems, thereby generalizing \cite[Proposition 10.11 and Theorem 10.12]{COP} and \cite[Theorem 5.2]{JKP}.

Let $(\CB,\CL,\theta,\CI_\af;\CJ)$ be a relative generalized Boolean dynamical system. Given a hereditary  $\CJ$-saturated ideal $\CH$ of $\CB$, we define
\[
\CB_{\CH}:=\{ A \in \CB :  [A] \in (\CB/\CH)_{reg} \}.	
\]
It is easy to see that  $\CB_\CH$ is an ideal of $\CB$ and $\CH \cup \CB_{reg} \subseteq \CB_\CH$.

Fix a hereditary $\CJ$-saturated ideal $\CH$ of $\CB$ and an ideal $\CS $  of $\CB_\CH$ such that $\CH \cup \CJ \subseteq \CS $. Note that $\CS$  is also an ideal of $\CB$. We let $I_{(\CH,\CS)}$ denote the ideal of $C^*(\CB,\CL,\theta, \CI_\af;\CJ):=C^*(p_A, s_{\af,B})$ generated by the family of projections (where we just write $\Delta_{[A]}$ for $\Delta^{(\CB/\CH,\CL,\theta)}_{[A]}$) \[
\biggl\{p_A-\sum_{\af \in \Delta_{[A]}} s_{\af,\theta_\af(A)}s_{\af,\theta_\af(A)}^*: A \in \CS\biggr\}.	
\]
Note that the family contains the family of projections $\{p_A: A \in \CH\}$ (because if $A\in\CH$, then $\Delta_{[A]}=\emptyset$). 

We put $p_{A, \CH}:=\sum_{\af \in \Delta_{[A]}} s_{\af,\theta_\af(A)}s_{\af,\theta_\af(A)}^*$ throughout this section.

\begin{lem} \label{Lemma:IHS}
The ideal $I_{(\CH,\CS)}$ is gauge-invariant and 
\begin{equation}\label{looks-ideal}
I_{(\CH,\CS)}=\overline{\operatorname{span}}\{s_{\af,B}(p_{A}-p_{A,\CH})s_{\bt,C}^*:A \in \CS,  \af,\bt \in \CL^*, B \in \CI_\af ~\text{and}~ C \in \CI_\bt \}.
\end{equation}
\end{lem}

\begin{proof} 
Since $I_{(\CH,\CS)}$ is generated by a set that is gauge-invariant, $I_{(\CH,\CS)}$ is gauge-invariant.

It follows from Lemma~\ref{mul-s*s} that the right-hand side $J$ of \eqref{looks-ideal} is an ideal of $C^*(\CB,\CL,\theta, \CI_\af)$. Since $s_{\af,B}(p_A-p_{A, \CH})s_{\bt,C}^*=s_{\af,B}p_As_{\bt,C}^*$ for all $A \in \CH$, $J$ contains the generators of $I_{(\CH,\CS)}$. Thus $I_{(\CH,\CS)} \subseteq J$. The opposite inclusion is clear. 
\end{proof}

We shall prove in Theorem \ref{isomorphism to quotient} that every gauge-invariant ideals of $C^*(\CB,\CL,\theta, \CI_\af;\CJ)$ is of the form $I_{(\CH,\CS)}$ for a hereditary $\CJ$-saturated ideal $\CH$ and an ideal $\CS$ of $\CB_\CH$ with $\CH \subseteq \CS$ and $\CJ \subseteq \CS $, and show that the quotient of $C^*(\CB,\CL,\theta, \CI_\af ;\CJ)$ by the ideal $I_{(\CH,\CS)}$ fall into the class of  $C^*$-algeras of  relative generalized Boolean dynamical systems. 
To do this, we first observe the following.

\begin{lem}\label{her-sat from ideal} 
Let $I$ be a nonzero ideal in $C^*(\CB,\CL,\theta, \CI_\af ;\CJ)$.
\begin{enumerate}
\item  The set $\CH_I:=\{A \in \CB: p_A \in I\}$
is a hereditary and $\CJ$-saturated ideal of $\CB$.
\item  The set 
\[
\CS_I:=\biggl\{A \in \CB_{\CH_I}: p_A -\sum_{\af \in \Delta_{[A]}}s_{\af,\theta_\af(A)}s_{\af,\theta_\af(A)}^* \in I\biggr\}	
\]
is an ideal of $\CB_{\CH_{I}}$ (and hence an ideal of $\CB$) with $\CH_I \subseteq \CS_I$ and $\CJ \subseteq \CS_I$. 
\end{enumerate} 
\end{lem}

\begin{proof}
(1): Suppose $A, B \in \CH_I$. Then $p_{A \cup B}=p_A+p_B-p_{A \cap B} \in I$, so $A \cup B \in \CH_I$.

Suppose then that $ B \in \CH_I$ and $A \in \CB$ with $A \subseteq B$. Then $p_A=p_{A \cap B}=p_Ap_B \in I$, so $A \in \CH_I$. 

This shows that $\CH_I$ is an ideal of $\CB$. To show that $\CH_I$ is hereditary, suppose $A \in \CH_I$. Then $p_As_{\af, \theta_\af(A)}=s_{\af, \theta_\af(A)} \in I$ for all $\af \in \CL$. Thus, $s_{\af,\theta_\af(A)}^*s_{\af,\theta_\af(A)}=p_{\theta_\af(A)} \in I$ for all $\af \in \CL$, that is, $\theta_\af(A) \in \CH_I$ for all $\af \in \CL$. 

Suppose now that $A \in \CJ$ and $\theta_\af(A) \in \CH_I$ for all $\af \in \Delta_{A}$. Then $s_{\af, \theta_\af(A)}= s_{\af, \theta_\af(A)}p_{\theta_\af(A)} \in I$ for all $\af \in \Delta_A$. Thus, $p_A= \sum_{\af \in \Delta_A}s_{\af,\theta_\af(A)}^*s_{\af, \theta_\af(A)} \in I$, which means that $A \in \CH_I$. Hence, the hereditary set $\CH_I$ is $\CJ$-saturated.

\vskip 0.5pc

(2): If $ A \in \CS_I$ and $B \in \CB_{\CH_I} $, then obviously $A \cap B \in \CB_{\CH_I}$, and
\begin{align*}
I \ni &~~ p_B \big(p_A -\sum_{\af \in \Delta_{[A]}}s_{\af,\theta_\af(A)}s_{\af,\theta_\af(A)}^*\big) \\
&=p_{A \cap B}-\sum_{\af \in \Delta_{[A]}}s_{\af,\theta_\af(A \cap B)}s_{\af,\theta_\af(A \cap B)}^* \\
&=p_{A \cap B}-\sum_{\af \in \Delta_{[A \cap B ]}}s_{\af,\theta_\af(A \cap B)}s_{\af,\theta_\af(A \cap B)}^*.
\end{align*} 
The last equality holds true since $p_Bs_{\af,\theta_\af(A)}=s_{\af,\theta_\af(A \cap B)}=0$ for $\af \in \Delta_{[A]}$ with $\theta_\af(A \cap B)=\emptyset$. 
Thus, $A \cap B \in \CS_I$. To show that $\CS_I$ is closed under finite unions, suppose $A, B \in \CS_I$. Then clearly, $A \cup B \in \CB_{\CH_I}$ and we see that 
\begin{align*} 
& p_{A \setminus B}-\sum_{\af \in \Delta_{[A \setminus B]}}s_{\af,\theta_\af(A \setminus B)}s_{\af,\theta_\af(A \setminus B)}^* \in I, \\
&p_{A \cap B}-\sum_{\af \in \Delta_{[A \cap B]}}s_{\af,\theta_\af(A \cap B)}s_{\af,\theta_\af(A \cap B)}^* \in I, \\
 & p_{B \setminus A}-\sum_{\af \in \Delta_{[B \setminus A]}}s_{\af,\theta_\af(B \setminus A)}s_{\af,\theta_\af(B \setminus A)}^* \in I.
\end{align*}
Adding the 3 elements above and using the fact that 
\begin{equation}\label{sum-division}
p_{A, \CH}=\sum_{\af \in \Delta_{[A \setminus B]}} s_{\af,\theta_\af(A \setminus B)}s_{\af,\theta_\af(A \setminus B)}^* +\sum_{\af \in \Delta_{[A \cap B]}} s_{\af,\theta_\af(A \cap B)}s_{\af,\theta_\af(A \cap B)}^*,
\end{equation} 
we see that $p_{A \cup B}-\sum_{\af \in \Delta_{[A \cup B]}}s_{\af,\theta_\af(A \cup B)}s_{\af,\theta_\af(A \cup B)}^* \in I$. Hence, $ A \cup B \in \CS_I$.

Since $\Delta_{[A]}=\emptyset$ for all $A \in \CH_I$, it is rather obvious that  $\CH_I \subseteq \CS_I$. For any $A \in \CJ \setminus \CH_I$, we have $p_A +I = \sum_{\af \in \Delta_{[A]}}s_{\af, \theta_\af(A)}s_{\af, \theta_\af(A)}^* +I$, and hence, $\CJ  \subseteq \CS_I$.
\end{proof}

\begin{prop}\label{isomorphism to quotient} 
Let $(\CB,\CL,\theta,\CI_\af;\CJ)$ be a relative generalized Boolean dynamical system. Suppose that $I$ is an ideal of $C^*(\CB,\CL,\theta, \CI_\af;\CJ)$. 

There is then a surjective $*$-homomorphism
\[
\phi_I:C^*(\CB/\CH_I,\CL,\theta,[\CI_\alpha];[\CS_I])\to
C^*(\CB,\CL,\theta,\CI_\alpha;\CJ)/I
\] 
such that
\[
\phi_I(p_{[A]_{\CH_I}})=p_A+I ~\text{ and}~	
\phi_I(s_{\alpha,[B]_{\CH_I}})=s_{\alpha,B}+I,
\]
where $[\CI_\af]:=\{[A] \in \CB / \CH_I: A \in \CI_\af \}$ and $[\CS_I]:=\{[A] \in \CB/ \CH_I:  A \in \CS_I\}$. Moreover, the following are equivalent.
\begin{enumerate}
\item $I$ is gauge-invariant.
\item The map $\phi_I$ is an isomorphism.
\item $I=I_{(\CH_I,\CS_I)}$.
\end{enumerate}
\end{prop}

\begin{proof} 
Since $p_A\in I$ and $s_{\alpha,A}=s_{\alpha,A}p_A\in I$ if $A\in \CH_I$, it follows that $\{p_A+I:A\in \CB\}\cup\{s_{\alpha,B}+I:\alpha\in\CL,\ B\in\CI_\alpha\}$ is a  $(\CB/\CH_I,\CL,\theta,[\CI_\alpha];[\CS_I])$-representation. So, the universal property of $C^*(\CB/\CH_I,\CL,\theta,[\CI_\alpha];[\CS_I])$ gives us a surjective $*$-homomorphism
\[
\phi_I:C^*(\CB/\CH_I,\CL,\theta,[\CI_\alpha];[\CS_I])\to
C^*(\CB,\CL,\theta,\CI_\alpha;\CJ)/I
\] 
such that
\[
\phi_I(p_{[A]_{\CH_I}})=p_A+I ~\text{ and}~	
\phi_I(s_{\alpha,[B]_{\CH_I}})=s_{\alpha,B}+I.
\]

That (1)$\implies$(2) follows from Theorem~\ref{GIUT for RGBDS}, and that (3)$\implies$(1) follows from Lemma~\ref{Lemma:IHS}.

(2)$\implies$(3): Since $I_{(\CH_I,\CS_I)}\subseteq I$, there is a surjective $*$-homomorphism $$q:C^*(\CB,\CL,\theta,\CI_\alpha;\CJ)/I_{(\CH_I,\CS_I)}\to C^*(\CB,\CL,\theta,\CI_\alpha;\CJ)/I$$ such that $q(p_A+I_{(\CH_I,\CS_I)})=p_A+I$ and $q(s_{\alpha,[B]_{\CH_{I}}}+I_{(\CH_I,\CS_I)})=s_{\alpha,[B]_{\CH_{I}}}+I$. An argument similar to the one used to construct $\phi_I$ gives us a surjective $*$-homomorphism $$\phi_{I_{(\CH_I,\CS_I)}}:C^*(\CB/\CH_{I},\CL,\theta,[\CI_\alpha];[\CS_{I}])\to C^*(\CB,\CL,\theta,\CI_\alpha;\CJ)/I_{(\CH_I,\CS_I)}$$ such that $\phi_{I_{(\CH_I,\CS_I)}}(p_{[A]_{\CH_{I}}})=p_A+I_{(\CH_I,\CS_I)}$ and $\phi_{I_{(\CH_I,\CS_I)}}(s_{\alpha,[B]_{\CH_{I}}})=s_{\alpha,B}+I_{(\CH_I,\CS_I)}$. We have that $\phi_I=q\circ \phi_{I_{(\CH_I,\CS_I)}}$. It follows that if $\phi_I$ is an isomorphism, then $q$ is an isomorphism and $I=I_{(\CH_I,\CS_I)}$.
\end{proof}

The set of pairs $(\CH,\CS)$, where $\CH$ is a hereditary $\CJ$-saturated ideal of $\CB$ and $\CS$ is an ideal of $\CB_\CH$ with $\CH \cup \CJ\subseteq \CS $ is a lattice with respect to the order relation defined by $(\CH_1,\CS_1)\le (\CH_2,\CS_2)\iff (\CH_1\subseteq \CH_2\land \CS_1\subseteq\CS_2)$. The set of gauge-invariant ideals of $C^*(\CB,\CL,\theta, \CI_\af;\CJ)$ is a lattice with the order given by set inclusion (the meet of $I_1$ and $I_2$ is the ideal $I_1\cap I_2$, and the join of $I_1$ and $I_2$ is the ideal generated by $I_1\cup I_2$).

\begin{thm}\label{gauge invariant ideal:characterization} Let $(\CB,\CL,\theta,\CI_\af;\CJ)$ be a relative generalized Boolean dynamical system. 
Then the map $(\CH,\CS) \mapsto I_{(\CH,\CS)} $ is a lattice isomorphism between the lattice of all pairs $(\CH,\CS)$, where $\CH$ is a hereditary  $\CJ$-saturated ideal of $\CB$ and $\CS$ is an ideal of $\CB_\CH$ with $\CH \cup \CJ\subseteq \CS $,  and the lattice of all  gauge-invariant ideals of $C^*(\CB,\CL,\theta, \CI_\af;\CJ)$.
\end{thm}

\begin{proof} 
It follows from Theorem~\ref{isomorphism to quotient} the map $(\CH,\CS) \mapsto I_{(\CH,\CS)} $ is onto.

To see the map $(\CH,\CS) \mapsto I_{(\CH,\CS)} $ is injective, we show that 
\begin{equation*}
\CH_{I_{(\CH,\CS)}}=\CH ~\text{and}~\CS_{I_{(\CH,\CS)}}=\CS.
\end{equation*}
It is easy to see that $\CH \subseteq \CH_{I_{(\CH,\CS)}}$ and $\CS \subseteq \CS_{I_{(\CH,\CS)}}$. It follows from the universal property of $\{p_A,s_{\alpha,B}:A\in\CB,\ \alpha\in\CL,\ B\in\CI_\alpha\}$ that there is a $*$-homomorphism $\phi:C^*(\CB,\CL,\theta,\CI_\alpha;\CJ)\to C^*(\CB /\CH, \CL,\theta, [\CI_\af];[\CS])$ such that $\phi(p_A)=p_{[A]}$ and $\phi(s_{\alpha,B})=s_{\alpha,[B]}$. Let $I'=\ker(\phi)$. Then $I_{(\CH,\CS)}\subseteq I'$. It follows from Corollary~\ref{cor:inj} that $\CH_{I'}=\CH$ and $\CS_{I'}=\CS$. It follows that $\CH_{I_{(\CH,\CS)}}\subseteq \CH_{I'}=\CH$ and $\CS_{I_{(\CH,\CS)}}\subseteq \CS_{I'}=\CS$.

Next we show that $(\CH,\CS) \mapsto I_{(\CH,\CS)} $ is order preserving. For this suppose that for $i=1,2$, $\CH_i$ is a hereditary $\CJ$-saturated ideal of $\CB$ and $\CS_i$ is an ideal of $\CB_{\CH_i}$ such that $\CH_i\cup\CJ\subseteq \CS_i$, and that $\CH_1\subseteq \CH_2$ and $\CS_1\subseteq\CS_2$. Suppose $A\in\CS_1$. We then have that $\theta_\alpha(A)\in \CH_2$ for $\alpha\in\Delta^{(\CB/\CH_1,\CL,\theta)}_{[A]}\setminus \Delta^{(\CB/\CH_2,\CL,\theta)}_{[A]}$, and thus that $p_A-\sum_{\af \in \Delta^{(\CB/\CH_1,\CL,\theta)}_{[A]}} s_{\af,\theta_\af(A)}s_{\af,\theta_\af(A)}^*\in I_{(\CH_2,\CS_2)}$. It follows that $I_{(\CH_1,\CS_1)}\subseteq I_{(\CH_2,\CS_2)}$.

For the converse suppose $I_1$ and $I_2$ are gauge-invariant ideals of $C^*(\CB,\CL,\theta,\CI_\alpha;\CJ)$ such that $I_1\subseteq I_2$. Then clearly, $\CH_{I_1}=\{A\in\CB:p_A\in I_1\}\subseteq \{A\in\CB:p_A\in I_2\}=\CH_{I_2}$. Suppose $A\in\CS_{I_1}$. Then $\theta_\alpha(A)\in \CH_{I_2}$ for $\alpha\in\Delta^{(\CB/\CH_{I_1},\CL,\theta)}_{[A]}\setminus \Delta^{(\CB/\CH_{I_2},\CL,\theta)}_{[A]}$. It follows that 
\begin{align*}
p_A -&\sum_{\af \in \Delta^{(\CB/\CH_{I_2},\CL,\theta)}_{[A]}}s_{\af,\theta_\af(A)}s_{\af,\theta_\af(A)}^*
=p_A -\sum_{\af \in \Delta^{(\CB/\CH_{I_1},\CL,\theta)}_{[A]}}s_{\af,\theta_\af(A)}s_{\af,\theta_\af(A)}^*\\
& + \sum_{\af \in \Delta^{(\CB/\CH_{I_1},\CL,\theta)}_{[A]}\setminus \Delta^{(\CB/\CH_{I_2},\CL,\theta)}_{[A]}}s_{\af,\theta_\af(A)}s_{\af,\theta_\af(A)}^* \in I_2.
\end{align*}
This shows that $\CS_{I_1}\subseteq\CS_{I_2}$. We thus have that $(\CH,\CS) \mapsto I_{(\CH,\CS)} $ is a lattice isomorphism. 
\end{proof}

\vskip 3pc

\subsection*{Acknowledgements}
This research was partially supported by Basic Science Research Program through the National Research Foundation of Korea(NRF) funded by the Ministry of Education(NRF-2017R1D1A1B03030540).

The authors would like to thank the referee for valuable comments and suggestions.

\end{document}